\documentclass[11pt]{amsart}

\usepackage{amsmath,amsthm,dsfont, txfonts,float}
\usepackage{graphicx,color}

\theoremstyle{definition}

\newtheorem{defn}{Definition}[section]

\newtheorem{eg}[defn]{Example}

\theoremstyle{remark}
\newtheorem{rmk}[defn]{Remark}

\theoremstyle{plain}

\newtheorem{cor}[defn]{Corollary}

\newtheorem{lem}[defn]{Lemma}

\newtheorem{thm}[defn]{Theorem}

\numberwithin{equation}{section}
\newcommand{\DLF}{DLF}
\newcommand{\DF}{\mathcal{E}}
\newcommand{\Hil}{\mathcal{H}}
\newcommand{\Mag}{\mathcal{M}}
\newcommand{\domDF}{\mathcal{F}}
\newcommand{\U}{\mathcal{U}}
\newcommand{\K}{\mathcal{K}}
\renewcommand{\L}{\mathcal{L}}
\newcommand{\R}{\mathcal{R}}

\DeclareMathOperator{\mult}{mult}
\DeclareMathOperator{\dom}{dom}
\DeclareMathOperator{\Tr}{Tr}

\title{Spectra of Magnetic Operators on the Diamond Lattice Fractal}
\author[Brzoska, Coffey, Rooney, Loew, Rogers]{Antoni Brzoska, Aubrey Coffey, Madeline Rooney,\\
 Stephen Loew, Luke~G. Rogers}

\thanks{Authors supported in part by the National Science Foundation through grant DMS-1262929.}

\subjclass[2000]{Primary 28A80, Secondary 60J35, 31E05, 47A07, 81Q10, 81Q35.}
\keywords{Analysis on Fractals, Sierpinski Gasket, Magnetic form, Schr\"{o}dinger operator}

\begin{document}

\begin{abstract}
We adapt the well-known spectral decimation technique for computing spectra of Laplacians on certain symmetric self-similar sets to the case of  magnetic  Schr\"odinger operators and work through this method completely for the diamond lattice fractal. This connects results of physicists from the 1980's, who used similar techniques to compute spectra of sequences of magnetic operators on graph approximations to fractals but did not verify existence of a limiting fractal operator, to recent work describing magnetic operators on fractals via functional analytic techniques. 
\end{abstract}

\maketitle

\section{Introduction}

This paper is motivated by the problem of understanding the properties of an electron confined to a fractal set in a magnetic field via the one-dimensional Peierls model.  Such problems have been extensively investigated in the physics literature using numerical techniques and renormalization group methods~\cite{Alexanderetal,Alexander,Rammal,AO,Rammal2,Ghez,Bel}.  Our goal is to give a rigorous mathematical model for this problem on a class of self-similar fractals, and to give a detailed analysis in the specific case of  the diamond lattice fractal (\DLF).  For simplicity of notation we work on the \DLF\ throughout the paper, though much of our approach is more general.  Specifically, we use various developments in fractal analysis~\cite{CS,IRT,ACSY,CSetal,HT,hinzetal,HR,HKMRS} to define a Schr\"{o}dinger operator based on a Laplacian intrinsic to the fractal.  In Section~\ref{sec:approx} we show that this operator can be approximated in a natural way using the self-similar structure of the fractal; this approximation is applicable more generally to resistance forms on self-similar fractals, see also~\cite{PostSimmer}. In particular it generalizes the technique introduced in~\cite{HKMRS} to calculate spectra for magnetic operators corresponding to fields that are locally exact in the setting of the Sierpinski Gasket.  We then show, in Section~\ref{sec:specdec} that the structure of the \DLF\ is such that the spectrum of the operator can be computed using a spectral decimation method~\cite{RammalToulouse,FukushimaShima,MT}. This type of method has previously been used to consider magnetic Schr\"{o}dinger operators on an infinite Sierpinski lattice, for which the numerically-obtained spectral data has good agreement with experimental results~\cite{Ghez}, however the existence of a limiting operator was not established in this setting until the recent work of Chen and Gyo~\cite{chen-guo}.  We then specialize to the case of a magnetic field that is uniform in the sense that the flux through a cell depends only on the scale of the cell (Section~\ref{sec:scalefield}) and conclude with some numerical results in this setting in Section~\ref{sec:numerics}.

\section{The Diamond Lattice Fractal}\label{sec:analysis}

\begin{figure}[htb]
\centering
\includegraphics[width=12cm]{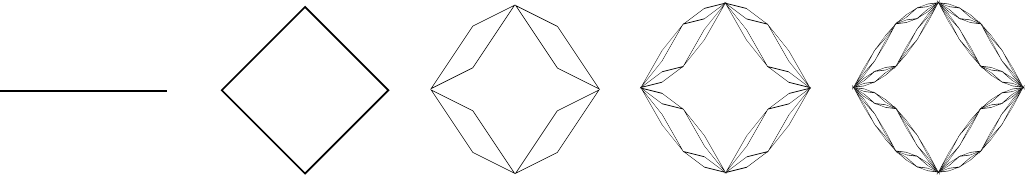}
\caption{Construction of the Diamond Lattice Fractal}\label{diamond}
\end{figure}

The Diamond Lattice Fractal (DLF) or Diamond Hierarchical Fractal is a particular case of the Berker lattice construction~\cite{BerkerOstlund}, and has been extensively studied in statistical physics (see, for example, \cite{LangloisTremblaySouthern, Alexanderetal, Derridaetal, Collet})  because the Migdal-Kadanoff renormalization is trivially exact in this setting. Mathematically rigorous versions of some statistical physics models are also understood, for example many fundamental results about percolation are proved in~\cite{HamblyKumagaiDiamond}.  We may realize it as a self-similar set $X\subset\mathbb{R}^2$ by introducing a scaling factor $s\leq1/8$ and maps
\begin{equation*}
	F_{j} \begin{bmatrix} x_1 \\x_2\end{bmatrix} =\frac{1}{2} \begin{bmatrix} 1&0\\(-1)^j & 4s\end{bmatrix} \begin{bmatrix} x_1 \\x_2\end{bmatrix}+ \frac{1}{\sqrt{2}} \begin{bmatrix} \cos \bigl( (2j-1)\pi/4\bigr) \\ \sin \bigl( (2j-1)\pi/4\bigr) \end{bmatrix} \quad \text{j=1,2,3,4,}
	\end{equation*}
and requiring that $X$ be the unique non-empty compact set so $X=\cup_{j=1}^{4}F_j(X)$.

We construct graphs that approximate $X$ in the manner  illustrated in Figure~\ref{diamond}.  Take $V_0=\{(-1,0),(1,0)\}$ to be the endpoints of the interval shown, and inductively let the scale~$m$ vertices be $V_m=\cup_{j=1}^4 F_j(V_{m-1})$.   For a word $w=w_1w_2\dotsm w_m\in\{1,2,3,4\}^{m}$ denote its length $m$ by $|w|$ and define $F_w=F_{w_1}\circ\dotsm\circ F_{w_m}$.  The edges of the scale $m$ graph are the images of the interval $[-1,1]\times\{0\}$ under the maps $F_w$ with $|w|=m$.  We write $x\sim_m y$ if there is $w$ with $|w|=m$ so $x=F_w(-1,0)$, $y=F_w(1,0)$.

Note that other treatments of the \DLF\ have not always defined it using the specific self-similarities $F_j$.  For much of our work this is of no significance because our analytic structure will depend on the graph structure of our approximations and associated electrical networks, for which the precise embedding into $\mathbb{R}^2$ is not relevant.  However we will later consider the notion of a uniform magnetic field through $X$, in which case it will be important that all cells of a given scale have the same size so that the flux, which is proportional to the area of the cell, depends only on its scale.  In particular we notice that the maps $F_j$ scale area by a factor of~$s$, so that the area enclosed by a scale $m$ cell is $2s^{m-1}$.

\subsection*{Resistance form and Laplacian on \DLF}

The crucial feature that permits us to do analysis on the \DLF\ is the existence of an irreducible local regular Dirichlet form $\DF$ and an associated non-positive definite, self-adjoint Laplacian operator $\Delta$ for which $\DF(f,g)=\int (-\Delta f)\bar{g}\,d\mu$, where $\mu$ is Hausdorff measure.  The existence and fundamental properties of such operators on fractals emerged in the probability and functional analysis literature, intially as a  mathematical treatment of physics models with anomalous diffusive behavior~\cite{Kusuoka85,BarlowPerkins,Kig89}, and subsequently as a subject of interest in its own right.  The monographs of Barlow~\cite{Barlowbook} and Kigami~\cite{Kigamibook} and the references therein give two standard approaches, but since the \DLF\ is only finitely ramified rather than post-critically finite we rely here upon Teplyaev's extension~\cite{Tep08} of Kigami's method. We note that the results given here about the Dirichlet form and Laplacian are not new: a direct approach that includes some estimates of the heat kernel is given in Section~4 of~\cite{HamblyKumagaiDiamond}. We also note that the harmonic structure on the \DLF\ is not regular in the sense explained in Chapter~3 of~\cite{Kigamibook}, so in particular the resistance metric completion of $V_\ast=\cup V_m$ is a strict subset of $X$. For this reason we will work with the Euclidean rather than the resistance metric. Though the following results are now standard we recall some salient features of the construction in order to fix notation.

Both the Dirichlet form and the Laplacian on the \DLF\ may be realized as a limits of corresponding objects on the the finite graph approximations in Figure~\ref{diamond}. Recall that the vertices of the scale~$m$ approximation are denoted $V_m$ and we write $x\sim_m y$ if there is an edge between $x$ and $y$. Define a sequence of graph Dirichlet forms and graph Laplacians by
\begin{gather}
	\DF_m(f) = \frac{1}{2} \sum_{x\in V_m}\sum_{y\sim_m x}  \bigl|f(x)-f(y)\bigr|^2 \label{eqn:defnofDFm}\\
	\Delta_m f(x) =  \frac{1}{\deg_m(x)} \sum_{y\sim_m x}  \bigl( f(y)-f(x)\bigr) \quad x\in V_m\setminus V_0 \notag
	\end{gather}
where $\deg_m(x)$ is the number of edges incident at $x$ in the scale $m$ graph.  Also define $\DF_m(f,g)$ by polarization and observe that if either $\sum_{y\sim_m x} \bigl(f(y)-f(x)\bigr)=0$ for $x\in V_0$  or $g=0$ on $V_0$  then
\begin{equation}\label{eqn:polarizedDF}
	 \DF_m(f,g) = \langle -\Delta_m f, g \rangle_{l^{2}(\mu_m)}
	\end{equation}
where $\mu_m$ is the measure on $V_m$ with mass $\deg_m(x)$ at $x\in V_m$, so $\mu_m(V_m)=2\cdot4^m$.

If $f$ is prescribed on $V_{m-1}$ then the extension to $V_m\setminus V_{m-1}$ that  minimizes~\eqref{eqn:defnofDFm} is obtained by setting $f$ on $F_w(V_1\setminus V_0)$ to be the average of the values in $f\circ F_w(V_0)$ for each word with $|w|=m-1$.  One then readily verifies that $\DF_m(f)=\DF_{m-1}(f)$, whence $\DF_m(f)$ is increasing in $m$.  When $\DF_m(f)$ has finite limit we write $f\in\domDF$ and call the limit $\DF(f)$.  If $\DF_m(f)$ is constant for all $m$ we call $f$ harmonic, and if it is constant for $m\geq n$ we call it harmonic at scale~$m$.  Functions in $\domDF$ can be approximated uniformly by functions harmonic at scale~$m$.  Note that $\DF_m(f,g)$ is independent of $m\geq n$ if $f$ is harmonic at scale~$n$.

Let $\mu$ be Hausdorff measure on $X$, scaled so $\mu(X)=2$.  By results of~\cite{Kigamibook,Tep08} the form $\DF$ is an irreducible local regular Dirichlet form on $L^2(\mu)$, so there is a self-adjoint Laplacian $\Delta$ for which 
\begin{equation}\label{eqn:defnofLap}
	\DF(f,g)=\int_X(-\Delta f)\bar{g}\,d\mu \quad\text{ for all } g\in\domDF_0,
	\end{equation}
where $\domDF_0=\{f\in\domDF:f|_{V_0}=0\}$ and we write $f\in\dom(\Delta)$ if there is a continuous $\Delta f$ for which~\eqref{eqn:defnofLap} is valid.  This is the Dirichlet Laplacian.

Just as $\DF_m(f)\to\DF(f)$ we have $4^m \Delta_m f\to\Delta f$ for $f\in\dom(\Delta)$.  To see this,
let $h_m^x$  be the scale~$m$ harmonic function that is $1$ at $x$ and $0$ at all other points of $V_m$, so $\int_{F_w(X)}h_m^x\,d\mu$ is independent of $w$ if $x\in F_w(X)$ and zero otherwise.  Thus $\int h_m^x\,d\mu=4^{-m}\deg_m(x)$.  One may uniformly approximate any continuous $h$ by the scale~$n$ harmonic functions $\sum_{x\in V_n} h(x)h_m^x$, from which
\begin{equation*}
	\int h\,d\mu 
	= \lim_m \int \sum_{x\in V_m} h(x)h_m^x \, d\mu
	= \lim_m  4^{-m} \sum_{x\in V_m} h(x)\deg_m(x)
	=  4^{-m} \int h\, d\mu_m
	\end{equation*}
and therefore $4^{-m}\mu_m$ converges weakly to $\mu$.  Then use~\eqref{eqn:polarizedDF} to see that when $4^m\Delta_m f$ converges uniformly on $X$ to $f'$ then $f'=\Delta f$ because
\begin{align*}
	\langle -\Delta f, g\rangle_{L^{2}(\mu)}
	=\DF(f,g)=\lim_m\DF_m(f,g)
	&= \lim_m \langle -\Delta_m f, g \rangle_{l^{2}(\mu_m)}\\
	&=\lim_m \int (-4^m\Delta_m f)\bar{g} \, (4^m d\mu_m)
	=\langle f',g\rangle_{L^{2}(\mu)}.
	\end{align*}

\subsection*{Magnetic Form and Magnetic Operator on the \DLF}

One may define differential forms on the \DLF\ and similar spaces using techniques from~\cite{IRT,ACSY,CSetal}.  We follow the construction in~\cite{IRT}, which provides a Hilbert space $\Hil$ of $1$-forms which is a module over $\mathcal{F}$ with $\|fa\|_\Hil\leq \|f\|_\infty\|a\|_\Hil$ if $f\in\domDF$, and  a derivation $\partial:\mathcal{F}\to\Hil$ such that $\|\partial f\|_{\Hil}^2=\DF(f)$ and the image is the space of exact forms.  A crucial result for our purposes is that one may define a magnetic form and self-adjoint magnetic operator in this setting.  We prove it using a variant of an argument from~\cite{HR}, employing the fact that $\Hil$ respects the cellular structure of $X$ and thus $\|a\|_\Hil^2=\sum_{|w|=n}\|a\mathds{1}_w\|_\Hil^2$, where $\mathds{1}_w$ is the characteristic function of $F_w(X)$.

\begin{thm}\label{thm:formisclosed}
For a real-valued $1$-form $a\in\Hil$  the quadratic form 
\begin{equation*}
	\DF^a(f)=\| (\partial+ia)f\|_{\Hil}^2
	\end{equation*}
 with domain $\mathcal{F}$ is closed on $L^2(\mu)$.  Thus there is a non-positive definite, self-adjoint, magnetic operator $\Mag^a_N$ satisfying
\begin{equation}\label{eqn:defnofMaga}
	\DF^a(f,g) = \langle -\Mag^a_N f,g\rangle_{L^2(\mu)}
	\end{equation}
for all $g\in\domDF$.  Moreover $\Mag^a_N$ has compact resolvent, hence its spectrum is a sequence $0\leq \lambda_1\leq\lambda_2\leq\dotsc$ accumulating only at $\infty$.
\end{thm}

\begin{proof}
According to  Lemma~4.2 in~\cite{HR} it suffices that there is $C=C(a)>0$ such that we  have a bound of the form
\begin{equation}\label{eqn:forformbound}
	\bigl\| fa\bigr\|_\Hil^2 \leq \frac{1}{2}\DF(f) + C \|f\|_{L^2(X,\mu)}^2
	\end{equation}
as this implies closedness of the quadratic form by the KLMN theorem and applicability of the Rellich criterion for the resolvent.

Observe that on the \DLF,  for $x,y$ in the cell $F_w(X)$ we have $|f(x)-f(y)|^2\leq \DF_w(f)$.  Write $f_w$ for the average (with respect to $\mu$) of $f$ over $F_w(X)$, so
\begin{gather}
	|f(x)-f_w|\leq \mu(F_w(X))^{-1}\int_{F_w(X)}|f(x)-f(y)|\,d\mu \leq \DF_w(f)^{1/2}, \quad\text{and hence} \notag\\
	|f(x)|^2 \leq 2\DF_w(f) + 2|f_w|^2 \leq 2\DF_w(f) + 2(|f|^2)_w  \label{eqn:formbound}
	\end{gather}
in which the last step uses Jensen's inequality.

Recalling $\|a\|_\Hil^2=\sum_{|w|=n}\|a\mathds{1}_w\|_\Hil^2$ take $n$ so large (depending only on $a$) that $\sup_{|w|=n}  \|a\mathds{1}_w\|_\Hil^2\leq \frac14$.  Then decompose $\|fa\|_\Hil^2$ according to cells of scale $n$ and compute using~\eqref{eqn:formbound} that
\begin{align*}
	\|fa\|_\Hil^2
	=\sum_{|w|=n} \|fa\mathds{1}_w\|_\Hil^2
	&\leq \sum_{|w|=n} \|f\mathds{1}_w\|_\infty \|a\mathds{1}_w\|_\Hil^2 \\
	&\leq 2\sum_{|w|=n} \bigl( \DF_w(f) + |f|^2_w \bigr) \|a\mathds{1}_w\|_\Hil^2 \\
	&\leq 2\Bigl(\sup_{|w|=n}  \|a\mathds{1}_w\|_\Hil^2\Bigr) \sum_{|w|=n} \DF_w(f) + 2\Bigl(\sup_{|w|=n} (|f|^2)_w\Bigr) \sum_{|w|=n}  \|a\mathds{1}_w\|_\Hil^2 \\
	&\leq \frac12 \DF(f) + 2\Bigl( \sup_{|w|=n} (|f|^2)_w\Bigr) \|a\|_\Hil^2
	\end{align*}
Using the fact that $\mu(X_w)=4^{-n}$ and crudely bounding $\int_{F_w(X)} |f|^2$ by $\|f\|_{L^2}^2$ we obtain~\eqref{eqn:forformbound} with $C=2\cdot4^n$, where $n$ depends only on $a$.  This shows the form is closed, and compactness of the resolvent follows as in Theorem~4.3 of~\cite{HR}.
\end{proof}

The above definition of $\Mag^a_N$ is the Neumann magnetic operator.  We can also define a Dirichlet magnetic operator $\Mag^a$ with the properties asserted in Theorem~\ref{thm:formisclosed}  by requiring~\eqref{eqn:defnofMaga} for all $f\in\domDF_0$, the subspace of functions vanishing on $V_0$.

The Neumann and Dirichlet magnetic operators are related by the magnetic normal derivative, which is defined for $f\in\dom(\Delta)$ and $x\in V_0$ by $d^af(x)=\lim_{m\to\infty} \DF(f,h_m^x)$.  This exists because $h_0^x-h_m^x\in\domDF_0$ and so $\DF(f,h_m^x)=\DF(f,h_0^x)+\int(\Delta f) (h_0^x-h_m^x)\,d\mu\to\int (\Delta f)h_0^x\,d\mu$ as $m\to\infty$.   Notice that then 
\begin{equation*}
	\DF(f,g) = - \int (\Delta f) g\, d\mu + \sum_{x\in V_0} g(x) d^a f.
\end{equation*}
In what follows we will most frequently study the Dirichlet operator, though the same techniques are applicable to the Neumann case.

In the next section we will see how the magnetic form and magnetic operator may be approximated by forms and operators on the graph approximants of the Diamond Lattice fractal.

\section{Approximation of Magnetic Forms and Operators}\label{sec:approx}

We have already seen that the Dirichlet form and Laplacian on the \DLF\ may be understood as limits of corresponding objects defined on the graph approximations.    Here we use results from~\cite{IRT} to show that the resistance structure of our self-similar space allows us to construct a sequence of magnetic operators and magnetic forms on the graphs that approximate the \DLF.  It should be noted that magnetic operators on graphs have been extensively studied, beginning with the work of~\cite{Sunada}, and there are generalizations to quantum graphs~\cite{KS}, but we will  only develop those aspects that are relevant for our needs.

 Recall that a function $f$ is harmonic if it minimizes $\DF(f)$ with prescribed values on $V_0$ and harmonic at scale $m$ if $f\circ F_w$ is harmonic for each $|w|=m$.  As in~\cite{IRT} we use the fact that $\DF$ is a resistance form to extend the module structure on $\Hil$ so as to allow multiplication by the characteristic function $\mathds{1}_w$ of a cell $F_w(X)$ and let $\Hil_m$ be the subspace of $\Hil_m$ spanned by elements $a_w\mathds{1}_w$ for $|w|=m$, where $a_w=\partial A_w$ for a function $A_w$ that is harmonic at scale $m$, so $a_w$ is exact at scale $m$.  This space is finite dimensional, hence closed, and we let $\Tr_m$ denote the projection $\Tr_m:\Hil\to\Hil_m$.    We will usually write $a_m=\Tr_m a$.  It is proved in~\cite{IRT} that $\cup_m \Hil_m$ is dense in $\Hil$.

Following~\cite{HKMRS} we identify $\Hil_m$ with the exact $1$-forms on the scale $m$ approximating graph.  An exact $1$-form on the $m$-scale graph is a function on directed edges such that the sum on the edges of any cell  $F_w(X)$ with $|w|=m$ is zero.  For $a_m\in\Hil_m$ we abuse notation to write $a_m=\sum_{|w|=m} a_w\mathds{1}_w$. Since $a_w=\partial A_w$ for  $A_w$  harmonic at scale $m$ and  unique modulo constants we can define a function on edges using the differences of the $A_w$ values.  For a directed edge $e_{xy}$ from $x$ to $y$ in the scale~$m$ graph we let $w'$ be the address of the unique $m$-cell $F_{w'}(X)$ containing this edge and treat $a_m$ as a function by writing $a_m(e_{xy})=A_{w'}(y)-A_{w'}(x)$.

A $1$-form on the graph approximation at scale $m$ defines a magnetic form and operator on this graph.  Let $a_m$ be as above and $f\in\domDF$. When $n\geq m$ let
\begin{gather}
	\DF_{n}^{a_{m}}(f) =\frac{1}{2} \sum_{x\in V_n} \sum_{y\sim_{n} x} \Bigl| f(x)- f(y)e^{i a_{m}(e_{xy})} \Bigr|^{2} \label{eqn:defnofDFnam}, \text{ and}\\
	\Mag^{a_{m}}_{n} f(x) = \frac{-1}{\deg_n(x)}\sum_{y\sim_{n}x} \Bigl( f(x)- f(y)e^{i a_{m}(e_{xy})} \Bigr) \quad\text{ for }x\in V_{n}\setminus V_{0}, \label{eqn:defnofMagam}
	\end{gather} 
and observe that if  $g$ %either $g$ or $\sum_{y\sim_{n}x} \bigl( f(x)- f(y)e^{i a_{m}(e_{xy})} \bigr)$
vanishes on $V_0$ then
\begin{equation} \label{eqn:DFmandMagfm}
	\DF_{n}^{a_{m}}(f,g) = \langle -\Mag^{a_{m}}_{n} f, g \rangle_{l^{2}(\mu_n)},
\end{equation}
where the measure $\mu_n$ has mass $\deg_n(x)$ at $x\in V_n$.

\subsection*{Gauge transformations and the structure of locally exact forms}

It is an important fact that on each $m$-cell the graph magnetic form $\DF^{a_m}_n$ may be obtained from the usual scale~$m$ resistance form by local gauge transformations (see Section~3 of~\cite{HKMRS}).   Specifically, we have from~\eqref{eqn:defnofDFnam} and the definition of $a_m$ that
\begin{align}
	\DF_{n}^{a_{m}}(f)
	&= \frac{1}{2} \sum_{|w|=m} \sum_{x,y\in F_w(V_0)}  \Bigl| f(x)- f(y)e^{i (A_{w}(y)-A_{w}(x))} \Bigr|^{2} \notag\\
	&=  \frac{1}{2}\sum_{|w|=m}\sum_{x,y\in F_w(V_0)}\Bigl| f(x)e^{iA_w(x)} - f(y)e^{i A_{w}(y)}  \Bigr|^{2} \notag\\
	&= \sum_{|w|=m}\DF_{n,w}(e^{iA_w}f) \label{eqn:DFamngauge}
	\end{align}
where $\DF_{n,w}$ means that we sum only over those edges in $F_w(X)$.

The same is true for $\DF^{a_m}(f)$, though the proof is slightly different.  Recall that $\DF^{a_m}(f)=\|(\partial+ia_m)f\|_\Hil^2$ and that $a_m=\sum_{|w|=m}a_w\mathds{1}_w$ with $a_w=\partial A_w$.  This and the decomposition of the Hilbert space according to cells (\cite{IRT} Theorem~4.6) implies
\begin{align}
	\bigl\|(\partial+ia_m)f\bigr\|_\Hil
	&=\Bigl\|  \sum_{|w|=m} \bigl( (\partial f)\mathds{1}_w  + i ( \partial A_w) f \mathds{1}_w  \bigr) \Bigr\|_\Hil \notag\\
	&= \sum_{|w|=m} \Bigl\|  e^{-iA_w} \partial \bigl( f e^{iA_w} \bigr) \mathds{1}_w  \Bigr\|_\Hil  \notag \\
	&=  \sum_{|w|=m} \Bigl\| \partial \bigl(  e^{iA_w} f \bigr) \mathds{1}_w  \Bigr\|_\Hil 
	= \sum_{|w|=m} \DF_w( e^{iA_w} f) \label{eqn:DFamgauge}
	\end{align}
where $\DF_w(h)= \DF(h\circ F_w)$.

\subsection*{Convergence of approximating forms}

The essential feature of the forms $\DF^{a_n}_n$ and the operators $\Mag^{a_m}_m$ is that they converge to $\DF^a$ and $\Mag^a$ respectively.  In the special case where the form has a local Coulomb gauge this was proved in~\cite{HKMRS}, but this assumption is essentially the same as assuming that the magnetic field has zero flux through all but finitely many holes, which is a serious constraint on the magnetic fields that can be considered, in particular precluding study of the fields of interest in the present work.  We will need the following more general result.

\begin{thm}\label{thm:cvgeofgraphapprox} If $a\in\Hil$ is real-valued and $f\in\domDF$ then
\begin{equation*}
	\DF^{a_n}_n(f)\to \DF^a(f).
	\end{equation*}
\end{thm}
\begin{proof}
Write
\begin{equation*}
	\bigl|  \DF^a(f) -\DF^{a_n}_n(f)  \bigr|
	\leq \bigl| \DF^{a_m}(f) -  \DF^a(f) \bigr|+ \bigl| \DF^{a_m}_n(f) -  \DF^{a_m}(f) \bigr| +\bigl| \DF^{a_n}_n(f) -  \DF^{a_m}_n(f) \bigr|.
	\end{equation*}
Lemma~\ref{lem:cvgelem1} shows the first term goes to zero as $m\to\infty$, Lemma~\ref{lem:cvgelem3} shows the same for the last term provided $n\geq m$, and Lemma~\ref{lem:cvgelem2} proves that the  middle term goes to zero as $n\to\infty$ for any fixed $m$, concluding the proof.
\end{proof}

\begin{lem}\label{lem:cvgelem1} For $a$ and $f$ as in Theorem~\ref{thm:cvgeofgraphapprox}, 
$\DF^{a_m}(f)\to\DF^a(f)$ as $m\to\infty$.
\end{lem}
\begin{proof}
Recall $\DF^{a}(f)=\|(\partial+ia)f\|_\Hil^2$ and similarly for $a_m$.  Thus
\begin{equation*}
	\bigl| \DF^{a_m}(f) - \DF^a(f) \bigr|=   \bigl| \|(\partial+ia_m)f\|_\Hil - \| (\partial+ia)f\|_\Hil\bigr| \bigl( \|(\partial+ia_m)f\|_\Hil +\| (\partial+ia)f\|_\Hil\bigr).
	\end{equation*}
However the first term is bounded by $\|(a-a_m)f\|_\Hil\leq\|f\|_\infty\|a-a_m\|_\Hil$.  For the second we use $\|a_m\|_\Hil\leq \|a\|_\Hil$ for all $m$ and direct estimation as follows:
\begin{equation*}
	\|(\partial+ia_m)f\|_\Hil \leq \|\partial f\|_\Hil + \|ia_m f\|_\Hil \leq (\DF(f))^{1/2}+ \|f\|_\infty \|a_m\|_\Hil \leq (\DF(f))^{1/2}+ \|f\|_\infty \|a\|_\Hil
	\end{equation*}
The same estimate is valid for $\|(\partial+ia)f\|_\Hil$, so we obtain
\begin{equation*}
	\bigl| \DF^{a_m}(f)- \DF^a(f)\bigr|
	\leq 2\bigl( (\DF(f))^{1/2}+\|f\|_\infty\|a\|_\Hil \bigr) \|f\|_\infty \|a-a_m\|_\Hil \xrightarrow{m\to\infty} 0.\qedhere
	\end{equation*}
\end{proof}

\begin{lem}\label{lem:cvgelem2}  For $a$ and $f$ as in Theorem~\ref{thm:cvgeofgraphapprox}, 
$\DF^{a_m}_n(f)\to\DF^{a_m}(f)$ as $n\to\infty$.
\end{lem}
\begin{proof}
For fixed $m$ we have $a_m\in\Hil_m$ so on each $m$-cell $F_w(X)$ there is $A_w$ such that $a_m(e_{xy})=A_w(y)-A_w(x)$ and by~\eqref{eqn:DFamngauge} and~\eqref{eqn:DFamgauge}
\begin{equation*}
\DF_{n}^{a_{m}}(f)
	= \sum_{|w|=m}\DF_{n,w}(e^{iA_w}f)
	\xrightarrow{n\to\infty}
	\sum_{|w|=m}\DF_{w}(e^{iA_w}f)
	= \DF^{a_m}(f) \qedhere
	\end{equation*}	
\end{proof}

\begin{lem}\label{lem:cvgelem3} For $a$ and $f$ as in Theorem~\ref{thm:cvgeofgraphapprox} 
\begin{equation*}
	\bigl| \DF^{a_m}_n(f) - \DF^{a_n}_n (f) \bigr| \quad\text{as $n\geq m\to\infty$.}
	\end{equation*}

\end{lem}
\begin{proof}
For $n\geq m$ we write $w'$ for words with $|w'|=m$ and $w$ for words with $|w|=n$.  Let $B_w$ be the functions on cells $F_w(X)$ such that~\eqref{eqn:DFamngauge} yields
\begin{equation*}
	\DF^{a_n}_n(f)= \sum_{|w|=n}\sum_{x,y\in F_w(V_0)} \Bigl|  f(x)e^{iB_{w}(x)} - f(y)e^{i B_{w}(y)} \Bigr|^{2}
\end{equation*}
For $\DF^{a_m}_n$ we instead have functions $A_{w'}$ on cells $F_{w'}(X)$ with $|w'|=m$, but $n\geq m$ so each $m$ cell is a union of $n$ cells, and we may write $A_w$ for the restriction of $A_{w'}$ to each $F_w(X)\subset F_{w'}(X)$.  Then from~\eqref{eqn:DFamngauge}
\begin{align*}
\DF^{a_m}_n(f)
	&= \sum_{|w'|=m}\sum_{\{|w|=n:F_w(X)\subset F_{w'}(X)\}} \sum_{x,y\in F_w(V_0)} \Bigl| f(x)e^{iA_{w}(x)}- f(y)e^{i A_{w}(y)} \Bigr|^{2}\\
	&= \sum_{|w|=n}\sum_{x,y\in F_w(V_0)}\Bigl| f(x)e^{iA_{w}(x)}- f(y)e^{i A_{w}(y)} \Bigr|^{2}
\end{align*}
We write the difference as a sum over $V_n$, with the value of $w$ in any term  implicitly given by the unique choice such that $x,y\in F_w(V_0)$.
\begin{align}
	\bigl| \DF^{a_m}_n(f) - \DF^{a_n}_n(f)\bigr|
	&= \Biggl| \sum_{x,y\in V_n, x\sim_n y}  \biggl( \Bigl| f(x)e^{iA_{w}(x)}- f(y)e^{i A_{w}(y)} \Bigr|^{2} - \Bigl|  f(x)e^{iB_{w}(x)} - f(y)e^{i B_{w}(y)} \Bigr|^{2} \biggr)\Biggr|\notag\\
	&\leq T_1^{1/2} T_2^{1/2} \label{eqn:CauchySchwarzforDFdiffs}
	\end{align}
where the terms $T_1$ and $T_2$ are estimated as follows, using the standard inequality $\DF_n(uv)\leq 2\|u\|_\infty^2\DF_n(v)+2\|v\|_\infty^2\DF_n(u)$.
\begin{align}
	T_1
	&=\sum_{x,y\in V_n, x\sim_n y} \Bigl( \bigl| e^{iA_w (x)}f(x)-e^{iA_w(y)}f(y) \bigr| - \bigl| e^{iB_w (x)}f(x)-e^{iB_w(y)}f(y) \bigr| \Bigr)^2 \notag\\
	&\leq \sum_{x,y\in V_n,  x\sim_n y} \Bigl| \bigl( e^{iA_w (x)}-  e^{iB_w (x)}\bigr) f(x) -\bigl( e^{iA_w(y)}  - e^{iB_w(y)}\bigr) f(y) \Bigr|^2 \notag\\
	&=\sum_{|w|=n} \DF_{n,w}\Bigl( \bigl( e^{iA_w}-  e^{iB_w}\bigr) f \Bigr) \notag\\ 
	&\leq 2\sum_{|w|=n} \Bigl(  \bigl\|  e^{iA_w}-  e^{iB_w} \bigr\|_{\infty}^2 \DF_{n,w}(f) +  \|f\|_{\infty}^2 \DF_{n,w}\bigl( e^{iA_w} -  e^{iB_w} \bigr)\Bigr)\label{eqn:estforfirstfactorindiffofsquares1}\\
	&=2\sum_{|w|=n} \Bigl( \|  A_w-B_w \|_{\infty}^2 \DF_{n,w}(f) +  \|f\|_{\infty}^2  \DF_{n,w}\bigl( e^{i(A_w-B_w)} \bigr)\Bigr)\label{eqn:estforfirstfactorindiffofsquares2}\\
	&\leq 2 \sup_{|w|=n} \| A_w-B_w \|_{\infty}^2 \sum_{|w|=n} \DF_{n,w}(f)  + 2 \|f\|_{\infty}^2 \sum_{|w|=n} \DF_{n,w}\bigl(A_w-B_w \bigr). \label{eqn:estforfirstfactorindiffofsquares}
	\end{align}
In passing from~\eqref{eqn:estforfirstfactorindiffofsquares1} to~\eqref{eqn:estforfirstfactorindiffofsquares2} we used  $\bigl|e^{iA_w(x)}-e^{iB_w(x)}\bigr|\leq|A_w(x)-B_w(x)|$ to estimate the left term, while on the right we used that $\DF_{n,w}(e^{i\theta} g)=\DF_{n,w}(g)$ and $\DF_{n,w}$ vanishes on constants.  From~\eqref{eqn:estforfirstfactorindiffofsquares2} to~\eqref{eqn:estforfirstfactorindiffofsquares} we applied 
\begin{equation*}
	\bigl| e^{i(A_w-B_w)(x)}- e^{i(A_w-B_w)(y)} \bigr|\leq \bigl| (A_w-B_w)(x)- (A_w-B_w)(y) \bigr|
	\end{equation*}
 and the definition of $\DF_{n,w}$.
\begin{align}
	T_2&= \sum_{x,y\in V_n} \Bigl( \bigl| e^{iA_w (x)}f(x)-e^{iA_w(y)}f(y) \bigr| + \bigl| e^{iB_w (x)}f(x)-e^{iB_w(y)}f(y) \bigr| \Bigr)^2 \notag\\
	&\leq 2 \sum_{x,y\in V_n} \Bigl( \bigl| e^{iA_w (x)}f(x)-e^{iA_w(y)}f(y) \bigr|^2 + \bigl| e^{iB_w (x)}f(x)-e^{iB_w(y)}f(y) \bigr|^2 \Bigr) \notag\\
	&=2\sum_{|w|=n} \Bigl( \DF_{n,w}\bigl( e^{iA_w}f \bigr) + \DF_{n,w}\bigl( e^{iB_w}f \bigr) \Bigr) \notag\\
	&\leq 4\sum_{|w|=n} \Bigl( \bigl\| e^{iA_w}\bigr\|_\infty^2 \DF_{n,w}(f) + \|f\|_\infty^2\DF_{n,w}\bigl( e^{iA_w} \bigr) +\bigl\| e^{iB_w}\bigr\|_\infty^2 \DF_{n,w}(f)+  \|f\|_\infty^2\DF_{n,w}\bigl( e^{iB_w} \bigr) \Bigr) \notag\\
	&\leq 4 \sum_{|w|=n} \Bigl( 2\DF_{n,w} (f) + \|f\|_\infty^2\Bigl( \DF_{n,w}\bigl( e^{iA_w} \bigr) + \DF_{n,w}\bigl( e^{iB_w} \bigr) \Bigr)\Bigr) \label{eqn:estforsecondfactorindiffofsquares1}\\
	&\leq 4 \sum_{|w|=n} \Bigl( 2\DF_{n,w} (f) + \|f\|_\infty^2\Bigl( \DF_{n,w}\bigl( A_w \bigr) + \DF_{n,w}\bigl( B_w \bigr) \Bigr)\Bigr) \label{eqn:estforsecondfactorindiffofsquares}
	\end{align}
where passage from~\eqref{eqn:estforsecondfactorindiffofsquares1} to~\eqref{eqn:estforsecondfactorindiffofsquares} uses $|e^{iA_w(x)}-e^{iA_w(y)}|\leq |A_{w}(x)-A_w(y)|$ and similarly for $B_w$.

Now $\sum_{|w|=n}\DF_{n,w}(f)=\DF_n(f)\leq \DF(f)$, and
\begin{equation*}
	\sum_{|w|=n}\DF_{n,w}(A_w)=\sum_{|w|=n}\|\partial A_w\mathds{1}_w\|_\Hil^2=\sum_{w=m} \|a_w\mathds{1}_w \|_{\Hil}^2=\|a_m\|_\Hil^2\leq \|a\|_\Hil^2
	\end{equation*}
and similarly $\sum_{|w|=n}\DF_{n,w}(B_w)=\|a_n\|_\Hil^2\leq \|a\|_\Hil^2$, so~\eqref{eqn:estforsecondfactorindiffofsquares} becomes
\begin{equation}\label{eqn:T2est}
	T_2\leq 8 \bigl(\DF(f) + \|f\|_\infty^2 \|a\|_{\Hil}^2 \bigr)
	\end{equation}
In the same way $\sum_{|w|=n}\DF_{n,w}(A_w-B_w)=\|a_n-a_m\|_\Hil^2$.  Combining this with the fact that for any $w$ with $|w|=n$
\begin{equation*}
	\|A_w-B_w\|_\infty^2\leq \DF_{n,w}(A_w-B_w)\leq \|a_n-a_m\|_{\Hil}^2
	\end{equation*}
the estimate~\eqref{eqn:estforfirstfactorindiffofsquares} is
\begin{equation}\label{eqn:T1est}
	T_1\leq 2 \bigl( \DF(f) + \|f\|_\infty^2\bigr)  \|a_n-a_m\|_\Hil^2
	\leq 2 \bigl( \DF(f) + \|f\|_\infty^2\bigr)  \|a -a_m\|_\Hil^2.
	\end{equation}
where the second inequality reflects the fact that $a_n$ is a sequence of projections of $a$ to nested subspaces $\Hil_n$.
Finally we have from~\eqref{eqn:CauchySchwarzforDFdiffs}, \eqref{eqn:T2est} and~\eqref{eqn:T1est}
\begin{equation*}
	\Bigl(\DF^{a_m}_n(f) - \DF^{a_n}_n(f)\Bigr)^2
	\leq T_1 T_2
	\leq  16\bigl(  \DF(f) + \|f\|_\infty^2  \bigr)\bigl( \DF(f) +\|f\|_\infty^2 \|a\|_\Hil^2 \bigr)\|a-a_m\|_\Hil^2.
	\end{equation*}
which establishes the result.
\end{proof}

\subsection*{Convergence of approximating magnetic operators}

We will need the following result, which is of a standard type.

\begin{thm}\label{thm:cvgeofmagops}
If $4^m \Mag^{a_m}_m f$ converges uniformly on $V_*\setminus V_0=\bigl(\cup_m V_m\bigr)\setminus V_0$ to a continuous function $F$ then $f\in\dom(\Mag^a)$ and $\Mag^a f$ is the continuous extension of $F$ to $X$.
\end{thm}

\begin{proof}
For $x\in V_m$ recall $h_m^x$ is the scale $m$ harmonic function which is $1$ at $x$ and vanishes on $V_m\setminus\{x\}$.   Given any $g\in\domDF$ with that vanishes on $V_0$ let $k_m(y)=4^m\sum_{x\in V_m} \Mag^{a_m}_m f(x) \bar{g}(x) h_m^x(y)$.  The integral $\int h_m^x(y)\,d\mu(y)=4^{-m}\deg_m(x)$, so
\begin{equation*}
	\int_X k_m(y)\, d\mu(y) = \sum_{x\in V_m} \deg_m(x) \Mag^{a_m}_m f(x) \bar{g}(x)
	= \langle \Mag^{a_m}_m f, g\rangle_{l^2(\mu_m)}
	=- \DF^{a_m}_m(f,g)
	\end{equation*}
where we were able to use~\eqref{eqn:DFmandMagfm} because $g=0$ on $V_0$. 
Now uniform convergence of $4^m \Mag^{a_m}_m f$ to $F$ and Theorem~\ref{thm:cvgeofgraphapprox} imply 
\begin{equation*}
	\int_X F(y)\bar{g}(y)\, d\mu(y) =-\DF^a(f,g)
	\end{equation*}
 for all $g\in \domDF$, from which $F=\Mag^{a}f$.
\end{proof}

Theorem~\ref{thm:cvgeofmagops} has a converse provided that the convergence of $a_n\to a$ is sufficiently uniform.

\begin{thm}\label{thm:findomMaga}
Suppose $a$ is such that $\sup_{|w|=m}4^m\bigl\|(a-a_m)\mathds{1}_{F_{w}(X)}\bigr\|_\Hil \to0$ as $m\to\infty$.  If $f\in\dom(\Mag^a)$ then $4^m \Mag^{a_m}_m f$ converges uniformly to $\Mag^a f$ on $V_*$.
\end{thm}
\begin{proof}
$\Mag^af\in\domDF$ it is  continuous, so for any approximate identity sequence $g_n^x$ at $x\in V_*$ we have $\langle \Mag^a f,g_n^x\rangle\to \Mag^a f(x)$.  If in addition $g_n^x\in\domDF$ then this implies $-\DF^a(f,g_n^x)\to\Mag^a f(x)$.  For $x\in V_*$ take $n_0$ so $x\in V_{n_0}$ and define $g_n^x$  for $n\geq n_0$ as follows. For each $n$-cell $F_w(X)$ containing $x$ take $A_w$ such that $\partial A_w=a_n$ as was done at the beginning of Section~\ref{sec:approx}. These functions are unique modulo constants; choose them such that $A_w(x)=0$ and let $A_n^x=A_w$ on $F_w(X)$.  Then, for all $n$ such that the denominator is non-zero, let 
\begin{equation}\label{eqn:gnx}
	g_n^x=\frac{4^n}{\deg_n(x)} e^{-iA_n^x}h_n^x.
	\end{equation}
This function is in $\domDF$ and supported on the $n$-cells that meet at $x$.   Moreover convergence of $a_n$ to $a$ ensures that $A_n^x$ is nearly constant on these cells when $n$ is large, or more precisely, $A_w\circ F_w$ converges to zero as $|w|\to\infty$.  This and the choice $A_n^x(x)=0$ ensures that $e^{-iA_n^x}\to 1$ as $n\to\infty$, and therefore that
\begin{equation*}
	\frac{4^n}{\deg_n(x)} \int_X e^{-iA_n^x}h_n^x\, d\mu
	=\frac{\int_X e^{-iA_n^x}h_n^x\, d\mu}{\int h_n^x\, d\mu }
	\to1 \quad\text{as }n\to\infty
	\end{equation*}
which establishes that $g_n^x$ is an approximate identity sequence from $\domDF$.  By direct computation we also have
\begin{equation}\label{eqn:Maganwhengisanharmonic}
	4^n\Mag^{a_n}_n f(x)
	= \langle \Mag^{a_n}_n f, g_n^x \rangle_{l^2(\mu_n)}
	= -\DF^{a_n}_n (f,g_n^x).
	\end{equation}
In light of the preceding the result follows from Lemma~\ref{lem:cvgeofgraphappwithgn}.
\end{proof}

\begin{lem}\label{lem:cvgeofgraphappwithgn} Under the hypotheses of Theorem~\ref{thm:findomMaga},
 $\DF^{a}(f,g_n^x)-\DF^{a_n}_n(f,g_n^x)\to0$  uniformly in $x$ as $n\to\infty$.
\end{lem}
\begin{proof}
The function $g_n^x$ is supported on the $n$-cells meeting at $x$ and $\partial A_n^x=a_n$ on these cells, so
\begin{equation}\label{eqn:gnxanharmonicity}
	\DF^{a_n}_n(f,g_n^x)
	= \DF_n\bigl( e^{iA_n^x}f, e^{iA_n^x}g_n^x\bigr)
	= \DF \bigl( e^{iA_n^x}f, e^{iA_n^x}g_n^x\bigr)
	\end{equation}
where in the last step we used the fact that $ e^{iA_n^x}g_n^x=4^n h_n^x/\deg_n(x)$ is harmonic at scale~$n$ by~\eqref{eqn:gnx}.  We may re-write this as
\begin{align*}
	\DF^{a_n}_n(f,g_n^x)
	= \left\langle  \partial \bigl( e^{iA_n^x}f\bigr) ,\partial\bigl( e^{iA_n^x}g_n^x\bigr) \right\rangle_{\Hil}
	&=  \left\langle e^{iA_n^x}(\partial+ia_n) f , e^{iA_n^x}(\partial +ia_n) g_n^x \right\rangle_{\Hil}\\
	&=\left\langle (\partial+ia_n) f , (\partial +ia_n) g_n^x \right\rangle_{\Hil}.
	\end{align*}
Subtracting this from $\DF^a(f,g_n^x)=\langle (\partial+ia)f,(\partial+ia)g_n^x\rangle$ we obtain
\begin{equation}\label{eqn:cvgeofgraphappwithgn1}
	\bigl| \DF^a(f,g_n^x) - \DF^{a_n}_n(f,g_n^x) \bigr| 
	\leq  \bigl| \langle (\partial +ia)f,i(a-a_n)g_n^x\rangle_{\Hil}\bigr| + \bigl| \langle i(a-a_n)f, (\partial+ia_n) g_n^x \rangle_{\Hil} \bigr|.
	\end{equation}
It is natural to decompose over the scale~$n$ cells $F_w(X)$ that meet at $x$, calling the corresponding set of words $W_n^x$ and to bound using Cauchy-Schwarz. For the first term we also use~\eqref{eqn:gnx} to see that $\|g_n^x\|_\infty=4^n/\deg_n(x)$, and obtain
\begin{align}
	\bigl| \langle (\partial +ia)f,i(a-a_n)g_n^x\rangle_{\Hil}\bigr|
	&\leq \bigl(\DF^a(f) \bigr)^{1/2} \sum_{w\in W_n^x} \|g_n^x \|_\infty \bigl\|(a-a_n) \mathds{1}_{F_w(X)} \bigr\|_\Hil \notag\\
	&\leq 4^n \bigl(\DF^a(f) \bigr)^{1/2} \sup_{|w|=n}  \bigl\|(a-a_n) \mathds{1}_{F_w(X)} \bigr\|_\Hil.\label{eqn:cvgeofgraphappwithgn2}
	\end{align}
The second term has one extra simplification, because the same reasoning as in~\eqref{eqn:gnxanharmonicity} shows that
\begin{equation*}
	\bigl\| (\partial +ia_n) g_n^x \bigr\|_\Hil
	= \bigl( \DF^{a_n} (g_n^x)\bigr)^{1/2} = \frac{4^n}{\deg_n(x)} \bigl( \DF (e^{iA_n^x}g_n^x) \bigr)^{1/2} = \frac{4^n}{\deg_n(x)} \bigl(\DF(h_n^x)\bigr)^{1/2}
	\end{equation*}
and on each cell $F_w(X)$ the contribution to $\DF(h_n^x)$ is $1$, so $ \bigl\| (\partial+ia_n) g_n^x \mathds{1}_{F_w(X)}\bigr\|_\Hil\leq 4^n/\deg_n(x)$.  Writing the same cellular decomposition as before we have
\begin{align}
	 \bigl| \langle i(a-a_n)f, (\partial+ia_n) g_n^x \rangle_{\Hil} \bigr|
	&\leq  \sum_{w\in W_n^x} \|f\|_\infty \bigl\|(a-a_n) \mathds{1}_{F_w(X)} \bigr\|_\Hil  \bigl\| (\partial+ia_n) g_n^x \mathds{1}_{F_w(X)}\bigr\|_\Hil \notag\\
	&\leq \frac{4^n\|f\|_\infty}{\deg_n(x)} \sum_{w\in W_n^x}\bigl\|(a-a_n) \mathds{1}_{F_w(X)} \bigr\|_\Hil  \notag\\
	&\leq 4^n \|f\|_\infty \sup_{|w|=n} \bigl\|(a-a_n) \mathds{1}_{F_w(X)} \bigr\|_\Hil, \label{eqn:cvgeofgraphappwithgn3}
	\end{align}
and then combining~\eqref{eqn:cvgeofgraphappwithgn1}, \eqref{eqn:cvgeofgraphappwithgn2} and~\eqref{eqn:cvgeofgraphappwithgn3} yields \begin{equation*}
	\bigl| \DF^a(f,g_n^x) - \DF^{a_n}_n(f,g_n^x) \bigr| 
	\leq 4^n  \Bigl( \bigl(\DF^a(f) \bigr)^{1/2} + \|f\|_\infty\Bigr)  \sup_{|w|=n} \bigl\|(a-a_n) \mathds{1}_{F_w(X)} \bigr\|_\Hil 
	\end{equation*}
whereupon the result follows by the hypothesis on the convergence of $a_n$ to $a$ made in Theorem~\ref{thm:findomMaga}.
\end{proof}

\section{Spectral Decimation on \DLF\ graphs}\label{sec:specdec}

In this section we show that for a special class of fields the spectrum and eigenfunctions of $\Mag^{a_m}_m$ are related to those of $\Mag^{a_{m-1}}_{m-1}$.  For this purpose it will be convenient to define $\Mag^{a_m}_n$ by~\eqref{eqn:defnofMagam} for all $x\in V_n$, not just those in $V_n\setminus V_0$.  We will do so in this section except when otherwise noted.

We begin by making a closer examination of the local structure of $\Mag^{a_m}_m$.  Our main result in this regard is~\eqref{eqn:celldecompofMagam3}, which is a decomposition of the operator into a sum over $(m-1)$-cells of gauge transformed copies of magnetic operators on $V_1$.  With this in hand we review some well-known results on spectral similarity and Schur complement and apply them to magnetic operators on $V_1$. The results suggest that $\Mag^{a_m}_m$ should be spectrally similar to  $\Mag^{a_{m-1}}_{m-1}$ if the fluxes through all cells of a given scale are the same.  We prove this in Theorem~\ref{thm:Specsimonelevel} using a gluing lemma for spectral similarity (Lemma~\ref{lem:gluing}) that generalizes a similar result from~\cite{MT}.

\subsection*{Local structure of Graph Magnetic Operators}

The gauge transformations introduced in the previous section correspond to conjugation by diagonal unitary transformations, at least locally. The simplest case occurs when the scale is zero.  For example, if $g=0$ on $V_0$ (Dirichlet boundary conditions) then for any $n$
\begin{align*}
	\langle -\Mag^{a_0}_n f,g\rangle_{l^2(\mu_n)}
	= \DF^{a_0}_n (f,g)
	&=\DF_n (e^{iA_0}f,e^{iA_0}g)\\
	&=\langle -\Delta_n  e^{iA_0}f, e^{iA_0}g \rangle_{l^2(\mu_n)}
	= \langle e^{-iA_0} (-\Delta_n) e^{iA_0}f, g \rangle_{l^2(\mu_n)},
	\end{align*}
so that $\Mag^{a_0}_n$ is obtained from  $\Delta_n$ by conjugation with the unitary diagonal transformation $e^{iA_0}$. In particular $\Mag^{a_0}_n$ and  $\Delta_n$  have the same eigenvalues and $f$ is an eigenvector of $\Mag^{a_0}_n$ if and only if $fe^{iA_0}$ is an eigenvector of $\Delta_n$.

For $\Mag^{a_m}_m$ the situation is more complicated because we have only  local gauge transformations.  We must therefore  conjugate by a different operator on each cell and the result is not globally unitary. Moreover when converting from $\DF^{a_m}_m$ to $\Mag^{a_m}_m$ we must keep track of the fact that each edge belongs to a unique cell, but the sum~\eqref{eqn:defnofMagam} defining $\Mag^{a_m}_m$ involves terms from more than one cell.

It is convenient to deal with this by introducing operators as follows. For a word $w$ with $|w|=m$ let
$R_w f=f\circ F_w$ map functions on $V_{m+n}$ to functions on $V_n$ and
\begin{equation}\label{eqn:defnofLw}
	L_{w,m+n} f (x)= \begin{cases}
	\frac{1}{\deg_{m+n}(x)} \bigl( f\circ F_w^{-1}(x)\bigr) &\text{if } x\in F_w(V_n) \\
	0 &\text{if } x\in V_{m+n}\setminus F_w(V_n)
	\end{cases}
	\end{equation}
map functions on $V_n$ to functions on $V_{m+n}$.  We will only need the cases $n=0,1$, and initially we look only at $n=0$.  Let $D=\Bigl(\begin{smallmatrix} 1&-1\\-1&1\end{smallmatrix}\Bigr)$ act on functions on $V_0$ and observe that 
\begin{equation*}
	-\DF_{m,w}(f,g)
	=\langle D(f\circ F_w),g\circ F_w\rangle_{l^2(\mu_0)}
	= \langle D R_w f, R_w g \rangle_{l^2(\mu_0)}
	= \langle L_{w,m} D R_w f,  g \rangle_{l^2(\mu_m)}
\end{equation*}
so from~\eqref{eqn:DFamngauge} when $g=0$ on $V_0$
\begin{align*}
	\langle \Mag^{a_m}_m f,g\rangle_{l^2(\mu_m)}
	= -\DF^{a_m}_m(f,g)
	&= \sum_{|w|=m} -\DF_{m,w}\bigl(e^{iA_w}f,e^{iA_w}g\bigr)\\
	&= \sum_{|w|=m} \left\langle L_{w,m} DR_w\bigl(e^{iA_w}f\bigr) , e^{iA_w}g\right\rangle_{l^2(\mu_m)}\\
	&= \sum_{|w|=m} \left\langle \bigl(T_{-A_w}L_{w,m} D R_wT_{A_w}\bigr)f , g \right\rangle_{l^2(\mu_m)}
	\end{align*}
where we have written $T_{A_w}$ for the operator of pointwise multiplication by $e^{iA_w}$.  This gives a cell decomposition of $\Mag^{a_m}_m$ at points in $V_m\setminus V_0$:
\begin{equation}\label{eqn:celldecompofMagam}
	\Mag^{a_m}_m = \sum_{|w|=m} T_{-A_w}L_{w,m} D R_wT_{A_w}.
	\end{equation}

This decomposition suggests breaking our magnetic field into pieces that act as gauge transformations at different scales.  Let $\K_m$  be the orthogonal complement  of $\Hil_{m-1}$ in $\Hil_m$ for $m\geq1$ and let $\K_0=\Hil_0$, so $\Hil_m=\oplus_0^m\K_j$.  Recall that $\cup_m \Hil_m$ is dense in $\Hil$, so $\Hil=\oplus_0^\infty\K_m$, and (abusing notation slightly) that each $\K_j$, $j\geq1$ may be decomposed as $\oplus_{|w|=j-1}\K_w$ where each $\K_w$ is isomorphic to $\K_1$ via the map $F_w$.  Using the identification of $\K_1$ with functions on the directed scale~$1$ graph we see that $\K_1$ is one-dimensional with basis element a non-exact $1$-form.  A symmetric such basis element is shown in Figure~\ref{fig:oneloopbasiselt}, as is a typical element of $\K_2$ obtained as a linear combination of copies of this basis element on the $2$-scale cells.  We call the symmetric basis element $k\in\K_1$ and let $k_w=k\circ F_w^{-1}$ be the corresponding basis element for $\K_w$.

\begin{figure}[htb]
\centering
\begingroup%
  \makeatletter%
  \providecommand\color[2][]{%
    \errmessage{(Inkscape) Color is used for the text in Inkscape, but the package 'color.sty' is not loaded}%
    \renewcommand\color[2][]{}%
  }%
  \providecommand\transparent[1]{%
    \errmessage{(Inkscape) Transparency is used (non-zero) for the text in Inkscape, but the package 'transparent.sty' is not loaded}%
    \renewcommand\transparent[1]{}%
  }%
  \providecommand\rotatebox[2]{#2}%
  \ifx\svgwidth\undefined%
    \setlength{\unitlength}{202.87297177bp}%
    \ifx\svgscale\undefined%
      \relax%
    \else%
      \setlength{\unitlength}{\unitlength * \real{\svgscale}}%
    \fi%
  \else%
    \setlength{\unitlength}{\svgwidth}%
  \fi%
  \global\let\svgwidth\undefined%
  \global\let\svgscale\undefined%
  \makeatother%
  \begin{picture}(1,0.44947548)%
    \put(0,0){\includegraphics[width=\unitlength,page=1]{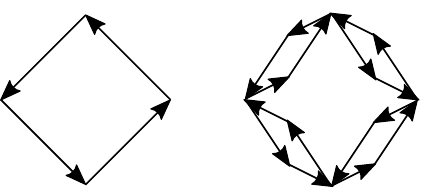}}%
    \put(-0.15402302,0.50525382){\color[rgb]{0,0,0}\makebox(0,0)[lt]{\begin{minipage}{0.49633036\unitlength}\raggedright \end{minipage}}}%
%    \put(0,0){\includegraphics[width=\unitlength,page=2]{Kfields.pdf}}%
    \put(0.070,0.33){\color[rgb]{0,0,0}\makebox(0,0)[lb]{\smash{$1$}}}%
    \put(0.06,0.08){\color[rgb]{0,0,0}\makebox(0,0)[lb]{\smash{$1$}}}%
    \put(0.31,0.08){\color[rgb]{0,0,0}\makebox(0,0)[lb]{\smash{$1$}}}%
    \put(0.31581774,0.33){\color[rgb]{0,0,0}\makebox(0,0)[lb]{\smash{$1$}}}%
    \put(0.58,0.08281537){\color[rgb]{0,0,0}\makebox(0,0)[lb]{\smash{$\alpha_3$}}}%
    \put(0.68248874,0){\color[rgb]{0,0,0}\makebox(0,0)[lb]{\smash{$\alpha_3$}}}%
    \put(0.73965035,0.09){\color[rgb]{0,0,0}\makebox(0,0)[lb]{\smash{$\alpha_3$}}}%
    \put(0.67,0.18){\color[rgb]{0,0,0}\makebox(0,0)[lb]{\smash{$\alpha_3$}}}%
    \put(0.82,0){\color[rgb]{0,0,0}\makebox(0,0)[lb]{\smash{$\alpha_4$}}}%
    \put(0.93,0.08281537){\color[rgb]{0,0,0}\makebox(0,0)[lb]{\smash{$\alpha_4$}}}%
    \put(0.855,0.19){\color[rgb]{0,0,0}\makebox(0,0)[lb]{\smash{$\alpha_4$}}}%
    \put(0.8,0.125){\color[rgb]{0,0,0}\makebox(0,0)[lb]{\smash{$\alpha_4$}}}%
    \put(0.95,0.30){\color[rgb]{0,0,0}\makebox(0,0)[lb]{\smash{$\alpha_1$}}}%
    \put(0.86,0.39){\color[rgb]{0,0,0}\makebox(0,0)[lb]{\smash{$\alpha_1$}}}%
    \put(0.78,0.325){\color[rgb]{0,0,0}\makebox(0,0)[lb]{\smash{$\alpha_1$}}}%
    \put(0.855,0.228){\color[rgb]{0,0,0}\makebox(0,0)[lb]{\smash{$\alpha_1$}}}%
    \put(0.72,0.28){\color[rgb]{0,0,0}\makebox(0,0)[lb]{\smash{$\alpha_2$}}}%
    \put(0.685,0.41){\color[rgb]{0,0,0}\makebox(0,0)[lb]{\smash{$\alpha_2$}}}%
    \put(0.58,0.31){\color[rgb]{0,0,0}\makebox(0,0)[lb]{\smash{$\alpha_2$}}}%
    \put(0.67,0.22){\color[rgb]{0,0,0}\makebox(0,0)[lb]{\smash{$\alpha_2$}}}%
  \end{picture}
\endgroup%

\caption{A graph $1$-form that spans $\K_1$ (left), and a typical element of $\K_2$ (right).}\label{fig:oneloopbasiselt}
\end{figure}
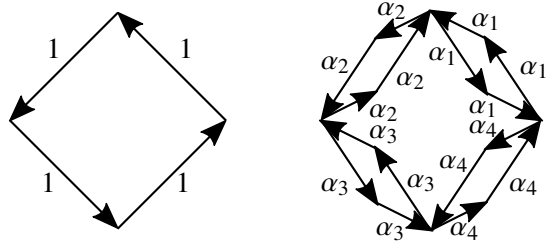

Note that if we  decompose  $\Mag^{\beta k}_1$ according to~\eqref{eqn:celldecompofMagam} then the directed graph function is $\beta$ on each edge, so the gauge operation on each edge is multiplication by $e^{iA}$ with $A=0$ at the source vertex of the directed edge and $A=\beta$ at the target vertex.  Then for each $j$, 
\begin{equation*}
	T_{-A_j}L_{j,1} DR_j T_{A_j}
	=L_{j,1} T_{-A} D T_{A} R_j
	= L_{j,1} \Bigl(\begin{smallmatrix} 1&-e^{i\beta} \\-e^{-i\beta} &1\end{smallmatrix}\Bigr) R_j
	\end{equation*}
Hence we may write a matrix for $\Mag^{\beta k}_1$ with the first two rows corresponding to the vertices in $V_0$ and the second two to those in $V_1\setminus V_0$ as follows,
\begin{equation}\label{eqn:Magbetak1}
	\Mag^{\beta k}_1 = \sum_j L_{j,1} T_{-A}DT_A R_j
	= \frac{1}{2}\begin{bmatrix}
	2 & 0 &   -e^{i\beta} & - e^{-i\beta} \\
	0 & 2 &  - e^{-i\beta} & -e^{i\beta} \\
	 -e^{-i\beta} & -e^{i\beta} & 2 & 0 \\
	 -e^{i\beta} & -e^{-i\beta}  & 0 &2 \\
	\end{bmatrix}
	\end{equation}
where the factor $\frac{1}{2}$ comes from $\deg_1(x)=2$ for all $x\in V_1$.

From the decomposition  $\Hil_m=\Hil_{m-1}\oplus_{|w|=m-1} \K_w$ write $a_m=a_{m-1}+ \sum_{|w|=m-1} \beta_w k_w $.  For each cell $F_{w}(X)$ with $|w|=m-1$ and each subcell $F_{wj}(X)$ we have harmonic functions $A_w$, $B_{wj}$ such that $a_{m-1}\circ F_w=\partial A_w$ and $\beta_w k_w\circ F_{wj}=\partial B_{wj}$.   The gauge map on $F_{wj}(X)$ is $T_{A_{wj}}=T_{A_w}T_{B_{wj}}$, so~\eqref{eqn:celldecompofMagam} becomes
\begin{equation}\label{eqn:celldecompofMagam2}
	\Mag^{a_m}_m
	=\sum_{|w|=m-1}  T_{-A_w} \Bigl( \sum_{j=1}^4 T_{-B_{wj}} L_{wj,m} DR_{wj} T_{B_{wj}}\Bigr) T_{A_w}.
	\end{equation}
However the same argument used in the computation of~\eqref{eqn:Magbetak1} shows that for each $w$ there is  $B_w$ such that $T_{-B_{wj}}L_{wj,m}=L_{wj,m}T_{-B_w}$ and $R_{wj}T_{B_{wj}}=T_{B_{w}}R_{wj}$ independent of $j$.  Moreover 
we can decompose
\begin{equation*}
	R_{wj}f=f\circ F_{wj}=f\circ F_w\circ F_j=R_j R_w f
	\end{equation*}
and for $x\in V_m$
\begin{equation*}
	L_{wj,m}f(x)= \frac{1}{\deg_m(x)} f \circ F_{wj}^{-1}
	=  \frac{2}{\deg_1(F_w^{-1}(x))\deg_{m}(x)} f \circ F_{j}^{-1}\circ F_{w}^{-1}
	= 2 L_{w,m} L_{j,1}.
	\end{equation*}
because all $y\in V_1$ have $\deg_1(y)=2$.

Note that in both of these expressions we are using the case $n=1$ of the definition of $R_w$ and $L_w$, meaning that they are considered as operators from functions on $V_{m}$ to functions on $V_1$ and conversely.

Using the above simplifications we conclude from~\eqref{eqn:celldecompofMagam2}  and~\eqref{eqn:Magbetak1} that
\begin{align}
	\Mag^{a_m}_m
	&= 2\sum_{|w|=m-1}  T_{-A_w} L_{w,m} \Bigl( \sum_{j=1}^4 L_{j,1} T_{-B_w} D T_{B_w}R_j \Bigr) R_w T_{A_w} \notag\\  
	&= 2\sum_{|w|=m-1}  T_{-A_w} L_{w,m} \Mag^{\beta_w k}_1 R_w T_{A_w}.\label{eqn:celldecompofMagam3}
	\end{align}
which decomposes $\Mag^{a_m}_m$ as a sum over $(m-1)$-cells of copies of magnetic operators on $V_1$, each of which is located at $F_w(V_1)\subset V_m$ and gauge transformed by $T_{A_w}$.

\subsection*{Spectral similarity}
The notion of spectral similarity we use is from~\cite{MT}, see also~\cite{RammalToulouse,FukushimaShima,Shima,MT2}, and is defined as follows.  

\begin{defn}
Let $\U$ and $\U_0$ be Hilbert spaces and  $U:\U_0\to\U$ be an isometry. Two bounded linear operators $M$ on $\U$ and $M_0\neq I$ on $\U_0$ are spectrally similar if there is a non-empty open $\Lambda\subset\mathbb{C}$ and functions $\phi_0,\phi_1:\Lambda\to\mathbb{C}$ such that
\begin{equation}\label{eqn:specsimdefn}
	U^* (M-z)^{-1} U = \bigl( \phi_0(z) M_0 - \phi_1(z) \bigr)^{-1}
	\end{equation}
at all $z\in\Lambda$.  If $\phi_0(z)\neq0$ we write $R(z)=\phi_1(z)/\phi_0(z)$ and note that
\begin{equation}\label{eqn:defnofR}
	U^* (M-z)^{-1} U = \frac{1}{\phi_0(z)} \bigl( M_0 - R(z) \bigr)^{-1}.
	\end{equation}
\end{defn}

By identifying $\U_0$ with the closed subspace  $U(\U_0)\subset\U$ one may characterize spectral similarity using the Schur complement.  This is done in~\cite{MT} by considering the case $\U_0\subset\U$, letting $\U_1$ denote its orthogonal complement and $P_j$ the projection $\U\to\U_j$ for $j=0,1$.  This permits a decomposition of $M$ into blocks
\begin{equation}\label{eqn:Masblocks}
	M=\begin{pmatrix} S & \tilde{X} \\ X & Q \end{pmatrix}
	\end{equation}
by setting $S= P_{0} M P_{0}$, $\tilde{X}= P_1 M P_0$,  $X = P_0 M P_1 $, $Q = P_1 M P_1$.  Note, too, that if $M=M^*$ then $\tilde{X}=X^*$.   With this notation, and writing  $\rho(A)$ for the resolvent set of an operator $A$, the following results are from 
Lemma~3.3, Corollary~3.4, Theorem~3.6 and Proposition~3.7 of~\cite{MT}.

\begin{thm}[{\protect \cite{MT}}]\label{thm:MT}\ 
\begin{enumerate}
\item For $z\in\rho(M)\cap\rho(Q)$,  $M$ and $M_0$ satisfy~\eqref{eqn:specsimdefn} if and only if
\begin{equation}\label{eqn:schurcondit}
 S-z - \tilde{X}(Q-z)^{-1}X = \phi_0(z)  M_0 - \phi_1(z) I
\end{equation}
\item If $M$ and $M_0$  satisfy~\eqref{eqn:specsimdefn} then $\phi_0$ and $\phi_1$ have unique analytic extensions to a connected component of $\rho(Q)$ and these extensions satisfy~\eqref{eqn:schurcondit}.  In the special case where the spectrum $\sigma(Q)$ of $Q$ consists only of isolated eigenvalues the $\phi_{j}(z)$ extend to be meromorphic on $\mathbb{C}$ with poles in $\sigma(Q)$.
\item If $M$ is spectrally similar to $M_0$ and $z\in\rho(Q)$ has $\phi_0(z)\neq0$ then
\begin{enumerate}
\item $R(z)\in\rho(M_0)$ if and only if $z\in \rho(M)$.
\item $R(z)$ is an eigenvalue of $M_0$ if and only if $z$ is an eigenvalue of $M$. There is a bijective map from the eigenspace of $M_0$ corresponding to $R(z)$ to the eigenspace of $M$ corresponding to $z$, with formula
\begin{equation}\label{eqn:efnextension}
	f_0 \mapsto f = \bigl( I-(Q-z)^{-1} X\bigr) f_0.
	\end{equation}
\end{enumerate}
\end{enumerate}
\end{thm}

A more precise analysis of what occurs in the case $\phi_0(z)=0$ or $z\not\in\rho(Q)$ is possible, see~\cite{BajorinChen}, but we will not need general results of this type because it will be easy to deal with these cases on the \DLF\  by direct arguments.

We illustrate the above notions with some computations that are extremely pertinent to the \DLF, namely the reduction of a magnetic operator on a diamond to an operator on a line segment.  Let $\U$ be the space of complex-valued functions on the vertices of the diamond and $\U_0$ be the space of functions on two opposite vertices, which we think of as a subspace of $\U$.

\begin{eg}\label{eg:simplefield}
 If the field through the hole in the diamond has magnitude $4\beta$ then the corresponding magnetic operator $M$ is that  in~\eqref{eqn:Magbetak1}, so may be written as in~\eqref{eqn:Masblocks} with
\begin{align*}
	S=Q= I, && X=\frac{-1}{2}\begin{bmatrix} e^{i\beta} &e^{-i\beta} \\ e^{-i\beta} & e^{i\beta} \end{bmatrix}, && \tilde{X}=X^*.
	\end{align*}

\begin{figure}[htb]
\centering
\begingroup%
  \makeatletter%
  \providecommand\color[2][]{%
    \errmessage{(Inkscape) Color is used for the text in Inkscape, but the package 'color.sty' is not loaded}%
    \renewcommand\color[2][]{}%
  }%
  \providecommand\transparent[1]{%
    \errmessage{(Inkscape) Transparency is used (non-zero) for the text in Inkscape, but the package 'transparent.sty' is not loaded}%
    \renewcommand\transparent[1]{}%
  }%
  \providecommand\rotatebox[2]{#2}%
  \ifx\svgwidth\undefined%
    \setlength{\unitlength}{12cm}%
    \ifx\svgscale\undefined%
      \relax%
    \else%
      \setlength{\unitlength}{\unitlength * \real{\svgscale}}%
    \fi%
  \else%
    \setlength{\unitlength}{\svgwidth}%
  \fi%
  \global\let\svgwidth\undefined%
  \global\let\svgscale\undefined%
  \makeatother%
  \begin{picture}(1,0.20840148)%
    \put(0,0){\includegraphics[width=\unitlength,page=1]{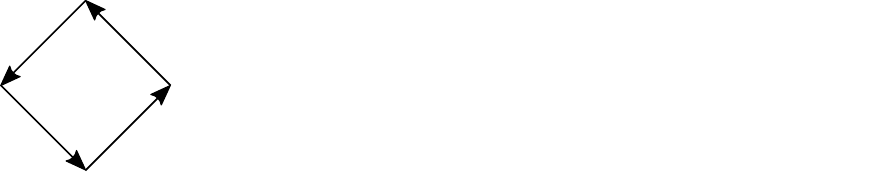}}%
    \put(0.028,0.155){\color[rgb]{0,0,0}\makebox(0,0)[lb]{\smash{$\beta$}}}%
    \put(0.023,0.031){\color[rgb]{0,0,0}\makebox(0,0)[lb]{\smash{$\beta$}}}%
    \put(0.153,0.031){\color[rgb]{0,0,0}\makebox(0,0)[lb]{\smash{$\beta$}}}%
    \put(0.158,0.155){\color[rgb]{0,0,0}\makebox(0,0)[lb]{\smash{$\beta$}}}%
    \put(0,0){\includegraphics[width=\unitlength,page=2]{simplefield.pdf}}%
    \put(0.348,0.108){\color[rgb]{0,0,0}\makebox(0,0)[lb]{\smash{0}}}%
    \put(0,0){\includegraphics[width=\unitlength,page=3]{simplefield.pdf}}%
    \put(0.545,0.155){\color[rgb]{0,0,0}\makebox(0,0)[lb]{\smash{$\beta+\alpha$}}}%
    \put(0.535,0.031){\color[rgb]{0,0,0}\makebox(0,0)[lb]{\smash{$\beta-\alpha$}}}%
    \put(0.715,0.031){\color[rgb]{0,0,0}\makebox(0,0)[lb]{\smash{$\beta-\alpha$}}}%
    \put(0.715,0.155){\color[rgb]{0,0,0}\makebox(0,0)[lb]{\smash{$\beta+\alpha$}}}%
    \put(0,0){\includegraphics[width=\unitlength,page=4]{simplefield.pdf}}%
    \put(0.89,0.108){\color[rgb]{0,0,0}\makebox(0,0)[lb]{\smash{$2\alpha$}}}%
  \end{picture}%
\endgroup%
\caption{Reducing two different magnetic fields on a single diamond}\label{fig:simplemagfield}
\end{figure}

The diagram on the left in Figure~\ref{fig:simplemagfield}  illustrates the field corresponding to $M$ by showing the $1$-form as a function on the directed edges.  The operator $M_0$ is just the discrete Laplacian on the two vertices, and the corresponding $1$-form has zero change on the edge.
\begin{equation*}
	M_0=\begin{bmatrix}1&-1\\-1&1\end{bmatrix}.
	\end{equation*}
Then
\begin{align*}
	 S-z - \tilde{X}(Q-z)^{-1}X
	& =\frac{1}{2(1-z)}\begin{bmatrix}
	2(1-z)^2 -1 & -\cos 2\beta \\
	-\cos 2\beta & 2(1-z)^2 -1
	\end{bmatrix}\\
	&= \frac{\cos 2\beta }{2(1-z)} M_0 - \frac{(-2z^2+4z- 1+\cos 2\beta)}{2(1-z)}I
	\end{align*}
so that~\eqref{eqn:schurcondit} holds, though the functions $\phi_j$  depend on the strength of the field.
\end{eg}

\begin{eg}\label{eg:lesssimplefield}
We also consider what occurs if we reduce a gauge equivalent field in the same manner.  The next simplest example of this kind is the function on directed edges shown on the right in Figure~\ref{fig:simplemagfield}, with a different difference along the top path than along the bottom. 
It  differs from Example~\ref{eg:simplefield} only in that 
\begin{equation*}
	X=\frac{-1}{2} \begin{bmatrix}e^{i(\beta+\alpha)}& e^{-i(\beta+\alpha)} \\ e^{-i(\beta-\alpha)}& e^{i(\beta-\alpha)} \end{bmatrix}.
	\end{equation*}
For this situation we compute
\begin{align*}
	 S-z - \tilde{X}(Q-z)^{-1}X
	& =\frac{1}{2(1-z)}\begin{bmatrix}
	2(1-z)^2 -1 & -e^{-2i\alpha} \cos 2\beta \\
	-e^{i2\alpha} \cos 2\beta & 2(1-z)^2 -1
	\end{bmatrix}\\
	&= \frac{\cos 2\beta}{2(1-z)} V M_0  V^* - \frac{(-2z^2+4z- 1+\cos 2\beta)}{2(1-z)}I
	\end{align*}
where $M_0$ is as before but
\begin{align*}
	V=\begin{bmatrix}
	e^{-i\alpha} &0 \\
	0& e^{i\alpha}
	\end{bmatrix}
	&&\text{ and therefore }
	&&V M_0 V^*  =\begin{bmatrix}
	1 & -e^{-i2\alpha} \\
	-e^{i2\alpha} & 1
	\end{bmatrix} 
	\end{align*}
is a gauge transform (by $V$) of $M_0$.  Notice that $V$ is the transform for the gauge field shown in the rightmost diagram in Figure~\ref{fig:simplemagfield}, which may be thought of as the net field obtained by tracing our original $1$-form to the two vertices of the unit interval.  As before, the functions $\phi_0$ and $\phi_1$ depend only on $z$ and the strength of the field through the hole, which in this case is $4\beta$.  The spectral similarity relation does not depend on the gauge field, all of which is accounted for in $V$.
\end{eg}

In combination with the gluing principle described next, the spectral similarity observed in Example~\ref{eg:lesssimplefield}  suggests that if the field has the same flux through all scale~$m$ holes then $\Mag^{a_m}_m$ should be spectrally similar to $\Mag^{a_{m-1}}_{m-1}$.  To prove this we need a result on gluing spectrally similar operators.

\subsection*{Gluing spectrally similar operators}

One of the most useful and well-known features of spectral similarity on self-similar graphs is the following fact: if there are operators $M^w$, each of which is spectrally similar to an operator $M^w_0$ via functions $\phi_0$ and $\phi_1$ that do not depend on~$w$, then there is a way to combine the $M^w$ such that the result is spectrally similar to $M^w_0$.  Moreover the way of combining the $M^w$ corresponds to a certain gluing operation on graphs.   The standard gluing lemma of this type is Lemma~3.10 of~\cite{MT}, but unfortunately it is not sufficient for our purposes.  Instead we prove the following closely related but more general result.  Note that in this lemma and its proof we simplify the notation by omitting many inclusion operators.

\begin{lem}\label{lem:gluing}
Let $\U=\U_0\oplus\U_1$ be a Hilbert space and $P_j$ denote projection onto $\U_j$.   Let $\U^w\subset\U$ be a collection of subspaces such that $\U=\sum_w\U^w$, define $\U^w_j=P_j\U_w$ for $j=0,1$, and let $P_j^w:\U^w\to\U^w_j$ denote the projections on these subspaces.
Suppose there are operators $M^w$ on $\U^w$ and $M^w_0$ on $\U^w_0$ that are spectrally similar with $U=U^*=P_0$ and functions $\phi_0$, $\phi_1$ that are independent of $w$  (and hence $M^w$ satisfies~\eqref{eqn:schurcondit} for each $w$).  If there are operators $\L^w:\U^w\to\U$ and $\R^w:U\to\U^w$ such that the following hypotheses hold:
\begin{enumerate}
\item \label{item:gluing1} For all $w$
\begin{align*}
	P_0 \L^w = \L^w P_0^w, &&  \R^wP_0 = P_0^w \R^w, &&  P_1 \L^w  =  \L^w P_1^w, &&  \R^w P_1=  P_1^w \R^w.
	\end{align*}
\item \label{item:gluing2}  For all $w$ and $w'\neq w$: $\R^w P_1 \L^w = P_1^w$ and $\R^{w'}P_1 \L^w=0$.
\item \label{item:gluing3} $P_0 = \sum_w \L^w P_0^w  \R^w$ and $P_1 = \sum_w \L^w P_1^w  \R^w$.
\end{enumerate}
Then $M=\sum_{w}\L^w M^w\R^w$ is spectrally similar to $M_0 = \sum_{w} \L^w  M_0^w  \R^w$ with the same functions $\phi_{0}(z)$ and $\phi_1(z)$.
\end{lem}
\begin{proof}
In light of~\eqref{eqn:schurcondit} we must compute $S-zP_0-\tilde{X}(Q-z)^{-1}X$.  First observe that for $j=0,1$
\begin{align}
	P_j (M-z)  P_j
	&= \sum_w \bigl( P_j \L^w M^w  \R^w P_j -  P_j \L^w  zP_j^w \R^w P_j \bigr) \notag\\
	&=  \sum_w  \L^w P_j^w  (M^w-z) P_j^w \R^w  \notag\\
	&=\begin{cases}
	\sum_w \L^w (S^w-z)  \R^w &\text{ if $j=0$}\\
	\sum_w \L^w (Q^w-z)  \R^w.&\text{ if $j=1$}
	\end{cases}
	\label{eqn:gluing1}
	\end{align}
where the first equality is by the assumption~\eqref{item:gluing3}, the third is from~\eqref{item:gluing1} and the fourth uses the definitions of $S^W$ and $Q^w$.  In the case $j=1$ using assumptions~\eqref{item:gluing2} and~\eqref{item:gluing3} we then compute from~\eqref{eqn:gluing1}
\begin{align*}
	P_1 (M-z) P_1 \Bigl( \sum_w  \L^w (Q^w-z)^{-1}  \R^w \Bigr)
	&= \sum_{w,w'} \L^{w'} (Q^{w'}-z)  \R^{w'} P_1 \L^w (Q^w-z)^{-1}  \R^w\\
	&= P_1 \sum_w \L^w P_1^w \R^w
	= P_1,
\end{align*}
from which 
\begin{equation}
	(Q-z)^{-1} = P_1(M-z)^{-1}P_1=\sum_w  \L^w (Q^w-z)^{-1}  \R^w. \label{eqn:gluing2}
	\end{equation}

We can compute $\tilde{X}=P_0MP_1$ and $X=P_1MP_0$ in a similar fashion:
\begin{gather}
	\tilde{X}
	=P_0 M P_1
	=  \sum_w P_0 \L^w M^w \R^w P_1
	= \sum_w \L^w  P_0^w M^w  P_1^w \R^w  
	= \sum_w \L^w \tilde{X}^w \R^w, \label{eqn:gluing3}\\
	X= P_1 M P_0
	=  \sum_w P_1 \L^w M^w \R^w P_0
	= \sum_w \L^w  P_1^w M^w  P_0^w \R^w  
	= \sum_w \L^w X^w \R^w. \label{eqn:gluing4}
	\end{gather}

Combining~\eqref{eqn:gluing3} for $\tilde{X}$, \eqref{eqn:gluing2} for $(Q-z)^{-1}$ and~\eqref{eqn:gluing4} for $X$ we obtain from assumption~\eqref{item:gluing2}
\begin{align}
	 \tilde{X}(Q-z)^{-1} X
	&= \sum_{w_1,w_2,w_3} \L^{w_1} \tilde{X}^{w_1} \R^{w_1} P_1 \L^{w_2} (Q^{w_2}-z)^{-1}  \R^{w_2} P_1 \L^{w_3} X^{w_3} \R^{w_3}\\
	&=\sum_{w} \L^w \tilde{X}^{w} P_1^w  (Q^{w}-z)^{-1} P_1^w  X^{w} \R^w\label{eqn:gluing5}
	\end{align}

Finally, using the case $j=0$ of~\eqref{eqn:gluing1}, which gives $S-zP_0$, and~\eqref{eqn:gluing5}
\begin{align*}
	S-zP_0- \tilde{X}(Q-z)^{-1} X
	&= \sum_w  \L^w (S^w-z)\R^w - \sum_{w} \L^w\tilde{X}^{w} P_1^w  (Q^{w}-z)^{-1} P_1^w  X^{w}\R^w\\
	&=\sum_w \L^w (S^w-z-\tilde{X}^w (Q^w-z)^{-1}X^w )\R^w\\
	&= \sum_w \L^w \bigl(\phi_{0}(z) M_0^w - \phi_1 (z)P_{0}^w \bigr)\R^w\\
	&= \phi_{0}(z) \Bigl(\sum_w \L^w M_0^w \R^w \Bigr)- \phi_1 (z)P_{0} \\
	&=\phi_0(z) M_0 - \phi_1(z)P_0
	\end{align*}
where the third equality uses the Schur characterization~\eqref{eqn:schurcondit} of the fact that $M^w$ is spectrally similar to $M^w_0$ on $\U_w$, and the fourth equality uses assumption~\eqref{item:gluing3}.  The final step is the definition of $M_0$ and we have proved, again from the Schur characterization, that $M$ and $M_0$ are spectrally similar via $\phi_0$ and $\phi_1$.
\end{proof}

\subsection*{Spectral similarity for graph magnetic operators}

Recall from~\eqref{eqn:celldecompofMagam3} that
 \begin{equation}\label{eqn:celldecompofMagam4}
	\Mag^{a_m}_m
	= 2\sum_{|w|=m-1}  T_{-A_w} L_{w,m} \Mag^{\beta_w k}_1 R_w T_{A_w}.
	\end{equation}
This is strongly reminiscent of the way in which the operators $M^w$ are glued to form $M$ in Lemma~\ref{lem:gluing}, with $M^w = \Mag^{\beta_w k_w}_1$ being the magnetic operator on $V_1$ corresponding to a flux of magnitude $4\beta_{w}$ through the hole in $F_w(V_1)$.
Moreover the computation in Example~\ref{eg:simplefield} shows that $M^w$ is spectrally similar to the usual Laplacian $D=\Bigl(\begin{smallmatrix} 1&-1\\-1&1\end{smallmatrix}\Bigr)$  on the unit interval (which in that example was denoted $M_0$) with functions $\phi_0$ and $\phi_1$ that depend only on the flux $4\beta_w$.   In order for Lemma~\ref{lem:gluing} to be applicable we would need that the flux depends only on the length $|w|=m-1$ of the word.  Accordingly we restrict to this class of magnetic fields. The $1$-form $k_w$ was defined in the paragraph following equation~\eqref{eqn:celldecompofMagam}.

\begin{defn}\label{def:fluxindepofscale}
We say the field $a$ has flux depending only on the scale if there is $a_0\in\Hil_0$ and  a sequence $\beta_m$ such that $a_m=a_{m-1}+\beta_m\sum_{|w|=m} k_w$ for all $m\geq1$.
\end{defn}

It should be noted that $\|k_w\|_\Hil$ is independent of $w$.  In fact it is easily checked that $\|k_w\|_\Hil=2$.  Moreover the $k_w$ were constructed so as to be an orthogonal set.  Using the fact that the number of $m$-cells is $4^{m-1}$
\begin{equation}\label{eqn:Hnormoffielddeponscale}
	\|a_m\|_\Hil^2 =  \sum_{n=1}^m \beta_n^2 \sum_{|w|=n} \|k_w\|_\Hil^2  = \sum_{n=1}^m 4^n \beta_n^2
	\end{equation}
and therefore we have the following.
\begin{lem}
For any sequence $\beta_m$ with $\sum_1^\infty 4^m \beta_m^2<\infty$ there is a field $a$ with flux independent of scale as in Definition~\ref{def:fluxindepofscale} and $\|a\|_\Hil^2=\sum_1^\infty 4^m \beta_m^2$.
\end{lem}

For this class of magnetic fields we may prove one of our main results.

\begin{thm}\label{thm:Specsimonelevel}
If $a\in\Hil$ is a real-valued $1$-form with flux depending only on the scale then $\Mag^{a_m}_m$ is spectrally similar to $\Mag^{a_{m-1}}_{m-1}$ via functions
\begin{align}\label{eqn:specdecfns}
 	\phi_0(z,\beta_m) = \frac{ \cos 2\beta_m }{2( 1-z)} && \phi_1(z,\beta_m) = \frac{(-2z^2+4z- 1+\cos 2\beta_m)}{2(1-z)}
\end{align}
\end{thm}
\begin{proof}
The proof is a direct application of Lemma~\ref{lem:gluing} and the computation in Example~\ref{eg:simplefield} to the expression~\eqref{eqn:celldecompofMagam4}.

Let $\U$ be functions on $V_m$ and decompose it as $\U_0\oplus\U_1$ where $\U_0$ is functions on $V_{m-1}\setminus V_0$ and $\U_1$ is functions on $V_m\setminus V_{m-1}$. For each $|w|=m-1$ let $\U^w$ be the subspace of functions on $F_w(V_1)$, $\U^w_0$ be functions on $F_w(V_0)$ and $\U^w_1$ be functions on $F_w(V_1\setminus V_0)$.  Define $\L^w=2T_{-A_w}L_{w,m}$ and $\R^w=R_wT_{A_w}$ so $\L^w:\U^w\to\U$ and $\R^w:\U\to\U^w$.  We verify the various conditions of Lemma~\ref{lem:gluing}.

Recall from~\eqref{eqn:defnofLw} that if $|w|=m-1$ and $n=1$ then
\begin{equation*}
	L_{w,m} f (x)= \begin{cases}
	\frac{1}{\deg_m(x)} \bigl( f\circ F_w^{-1}(x)\bigr) &\text{if } x\in F_w(V_1) \\
	0 &\text{if } x\in V_{m}\setminus F_w(V_1)
	\end{cases}
	\end{equation*}
from which it follows easily that
\begin{equation}\label{eqn:P0Lw}
	P_0\L^w f(x)
	=\left\{ \begin{aligned}
	 &\textstyle \frac{2e^{-iA_w(x)}}{\deg_m(x)} \bigl( f\circ F_w^{-1}(x)\bigr) &&\text{if } x\in F_w(V_0) \\
	&0 &&\text{if } x\in V_{m}\setminus F_w(V_0)
	\end{aligned} \right\}\\
	=\L^w P^w_0
	\end{equation}
and
\begin{align}
	P_1\L^wf(x)
	&=\begin{cases}
	 \frac{2e^{-iA_w(x)}}{\deg_m(x)} \bigl( f\circ F_w^{-1}(x)\bigr) &\text{if } x\in F_w(V_1\setminus V_0) \\
	0 &\text{if } x\in V_{m}\setminus F_w(V_1\setminus V_0)
	\end{cases} \notag\\
	&=\begin{cases}
	 e^{-iA_w(x)} \bigl( f\circ F_w^{-1}(x)\bigr) &\text{if } x\in F_w(V_1\setminus V_0) \\
	0 &\text{if } x\in V_{m}\setminus F_w(V_1\setminus V_0)
	\end{cases} \label{eqn:P1Lw}\\
	&=\L^w P^w_1.\notag
	\end{align}
because $x\in F_{w}(V_1\setminus V_0)$ implies $\deg_m(x)=2$. 
Similarly one sees from $R_wf=f\circ F_w$ that
\begin{align}
	\R^w P_0 f(y) 
	=\left\{ \begin{aligned}
	&e^{iA_{w}(F_w(y))}  f\circ F_w (y) &&\text{if } y\in V_0\\
	&0 &&\text{otherwise}
	\end{aligned}\right\}
	=P^w_0\R^w \label{eqn:RwP0}\\
	\R^w P_1 f(y)
	=\left\{ \begin{aligned}
	& e^{iA_{w}(F_w(y))}  f\circ F_w (y)  &&\text{if } y\in V_1\setminus V_{0}\\
	&0 &&\text{otherwise}
	\end{aligned} \right\}
	=P^w_1\R^w \label{eqn:RwP1}
	\end{align}
and from~\eqref{eqn:P1Lw} and~\eqref{eqn:RwP1} with $x=F_w(y)$ we have $\R^wP_1\L^w=P^w_1=\L^w P^w_1\R^w$ and $\R^{w'}P_1\L^w=0$ for $w'\neq w$.   Since the sets $F_w(V_1\setminus V_0)$  do not intersect and have union $V_m\setminus V_{m-1}$ this also establishes $\sum_w \L^w P^w_1\R^w=\sum_w P^w_1 = P_1$.

Lastly, consider the sum $\sum_w \L^w P^w_0\R^w$. From~\eqref{eqn:P0Lw} and~\eqref{eqn:RwP0} it is precisely
\begin{equation*}
	\sum_{|w|=m-1} \L^w P^w_0\R^w
	=\left\{\begin{aligned}
	&\textstyle\sum_{|w|=m-1} \frac{2}{\deg_m(x)} f(x) &&\text{if } x\in V_{m-1}\\
	&0 &&\text{otherwise}
	\end{aligned}\right\}
	=P_0
	\end{equation*}
because the sum yields the number of $m-1$ cells that meet at $x$ and each $m-1$ cell intersecting $x$ contains two $m$-cells that meet at $x$ and contribute to $\deg_m(x)$.

Now~\eqref{eqn:celldecompofMagam4} is $\sum_{|w|=m-1} \L^w M^{\beta_m k}_1 \R^w$.  We know from  the computation in Example~\ref{eg:simplefield} that $\Mag^{\beta_m k}_1$ is spectrally similar to $D$ via the functions specified in~\eqref{eqn:specdecfns}, so from Lemma~\ref{lem:gluing} $\Mag^{a_m}_m$ is spectrally similar to 
\begin{align*}
	\sum_{|w|=m-1} \L^w D \R^w
	&=2\sum_{|w|=m-1} T_{-A_w}  L_{w,m} D R_w T_{A_w}\\
	&=\sum_{|w|=m-1} T_{-A_w} L_{w,m-1} D R_w T_{A_w}
	=\Mag^{a_{m-1}}_{m-1}
	\end{align*}
by~\eqref{eqn:celldecompofMagam} and the observation that $2L_{w,m}=L_{w,m-1}$ on functions on $V_{m-1}$  because $\deg_m(x)=2\deg_{m-1}(x)$ for points in this set.
\end{proof}

Let $\mult_m(z)\in\mathbb{N}$ be the multiplicity of $z$ as an eigenvalue of $\Mag^{a_m}_m$, with the convention that $\mult_m(z)=0$ if $z$ is not an eigenvalue of $\Mag^{a_m}_m$.  As in~\eqref{eqn:defnofR} for the specfic functions $\phi_0$, $\phi_1$ from~\eqref{eqn:specdecfns} define
\begin{equation}\label{eqn:RmapforDLF}
	R_m(z) =  \frac{-2z^2+4z- 1+\cos 2\beta_m}{\cos2\beta_m}.
\end{equation}
Using Theorem~\ref{thm:MT} and some elementary computations we can determine the spectrum of $\Mag^{a_m}_m$ from that of $\Mag^{a_{m-1}}_{m-1}$. In the following result we restrict $\Mag^{a_m}_m$ to $V_m\setminus V_0$ so as to obtain the result for the Dirichlet operator.

\begin{cor}\label{cor:spectralmultiplicities}
For the Dirichlet magnetic operator $\Mag^{a_m}_m$ on $V_m\setminus V_0$ we have
\begin{enumerate}
\item  $\mult_m(1)=\frac{1}{3}(4^{m}+2)$.
\item   If $\cos2\beta_m=0$ and $z_\pm=1\pm\frac{1}{\sqrt{2}}$ then $\mult_m(z_\pm)=\frac{2}{3}(4^{m-1}-1)$.
\item  If $\cos2\beta_m\neq0$ and $z\neq1$ is an eigenvalue of $\Mag^{a_m}_m$, then $R_m(z)$ is an eigenvalue of $\Mag^{a_{m-1}}_{m-1}$ and $\mult_{m}(z)=\mult_{m-1}(z)$.
\end{enumerate}
\end{cor}
\begin{proof}
The number of vertices in $V_m$ is $\frac{2}{3}(4^m+2)$, so the dimension of the matrix $\Mag^{a_m}_m$ on $V_m\setminus V_0$  is  $\frac{2}{3}(4^m-1)$.  The difference between this and the dimension of $\Mag^{a_{m-1}}_{m-1}$ is $2\cdot4^{m-1}$.

If $z=1$ then $\Mag^{a_m}_m-z=\Big(\begin{smallmatrix}0&\tilde{X}\\ X&0\end{smallmatrix}\Bigr)$ and the rank of $\tilde{X}$ cannot exceed $\frac{2}{3}(4^{m-1}-1)$, so that 
\begin{equation*}
	\mult_{m}(1)
	\geq 2\cdot4^{m-1}-\frac{2}{3} (4^{m-1}-1)
	= \frac{1}{3}(4^{m}+2).
	\end{equation*}

If $\cos2\beta_m=0$ then the Schur complement~\eqref{eqn:schurcondit} is $(1-z-\frac{1}{2(1-z)})I$, so the eigenvalues are $z_\pm=1\pm\frac{1}{\sqrt{2}}$.  Then by~\eqref{eqn:efnextension} in Theorem~\ref{thm:MT} any function $f_0$ on $V_{m-1}\setminus V_0$ can be extended by $(I-(Q-z_\pm)^{-1}X)f_0$ to an eigenfunction of $\Mag^{a_m}_m$ with eigenvalue $z_\pm$, so $\mult_m(z_\pm)\geq \frac{2}{3}(4^{m-1}-1)$.  Hence
\begin{multline}\label{eqn:multiplicitiescomparison}
	  2\Bigl(\frac{2}{3}\Bigr)(4^{m-1}-1)  + \frac{1}{3}(4^{m}+2)   \\
	\quad \leq \mult_m(p+)+\mult_m(p_-)+\mult_m(1) 	\leq \frac{2}{3}(4^m-1)
	\end{multline}
so that both inequalities are equalities and the multiplicities are as stated in~(1) and~(2).

If, on the other hand, $z\neq1$ and $\cos2\beta_m\neq0$ then the bijection~\eqref{eqn:efnextension} implies that for each eigenvalue $w$ of $\Mag^{a_{m-1}}_{m-1}$ we have $\mult_m(z)=\mult_{m-1}(w)$  for any $z$ such that $R_m(z)=w$.  Moreover there are two $z$ values with $R_m(z)=w$ because we have assumed $z\neq1$, which is the critical point of $R_m$.  Now $\sum_w \mult_{m-1}(w)$, summing over $w$ in the spectrum of $\Mag^{a_{m-1}}_{m-1}$, is $\frac{2}{3}(4^{m-1}-1)$, so $\sum_z \mult_m(z)$ over those $z$ so that $R_m(z)$ is an eigenvalue of $\Mag^{a_{m-1}}_{m-1}$ is $\frac{4}{3}(4^{m-1}-1)$, and the same computation as in~\eqref{eqn:multiplicitiescomparison} implies these and the eigenvalue $1$ comprise the spectrum of $\Mag^{a_m}_m$.
\end{proof}

From Corollary~\ref{cor:spectralmultiplicities} the spectrum of $\Mag^{a_m}_m$ can be computed using sequences of preimages under the maps $R_m$.  If $\mathcal{R}_{k,m}=R_{k}\circ\dotsm\circ R_m$ for $1\leq k\leq m$ we may describe it as follows.

\begin{cor}\label{cor:Magamspectrum}
The spectrum of $\Mag^{a_m}_m$ is $\{1\} \cup \Bigl(\cup_{k=2}^{m} \mathcal{R}_{k,m}^{-1}\{1\}\Bigr)$, with the multiplicity of values in $\mathcal{R}_{k,m}^{-1}\{1\}$ being $\frac{4}{3}(4^{k}-1)$ and $\mult_m(1)=\frac{1}{3}(4^{m}+2)$.
	\end{cor}
\begin{proof}
A direct application of Corollary~\ref{cor:spectralmultiplicities} shows that when $m>1$ the spectrum of $\Mag^{a_m}_m$ satisfies
\begin{equation*}
	\sigma\bigl(\Mag^{a_m}_m\bigr)
	= \{1\} \cup R_{m}^{-1}\bigl(\sigma\bigl(\Mag^{a_{m-1}}_{m-1}\bigr)\bigr).
	\end{equation*}
This is even true when $\cos 2\beta_m=0$, because then $\mathcal{R}_m^{-1}\sigma\bigl(\Mag^{a_{m-1}}_{m-1}\bigr)=z_\pm$ with both multiplicities equal to the number of points in $\sigma\bigl(\Mag^{a_{m-1}}_{m-1}\bigr)$, which is $\frac{2}{3}(4^{m-1}-1)$.

The result then follows by induction and the fact that $\sigma\bigl(\Mag^{a_{0}}_{0}\bigr)$ is empty.  The other  multiplicities are from Corollary~\ref{cor:spectralmultiplicities}.
\end{proof}

The corresponding eigenfunctions may be found by iterated application of~\eqref{eqn:efnextension}, or in the case $\cos2\beta_m=0$, applying the extension map to any function on $V_{m-1}$.  We note that applying~\eqref{eqn:efnextension}  at level $V_m$ does not change the function on $V_{m-1}$, so it is immediate that the sequence of functions obtained converges on $V_\ast=\cup_m V_m$.

\section{Spectrum of $\Mag^a$ for a field depending only on the scale}\label{sec:scalefield}

Theorems~\ref{thm:cvgeofmagops} and~\ref{thm:findomMaga} describe circumstances under which the spectrum of $\Mag^a$ can be computed using the spectra of the graph operators $\Mag^{a_m}_m$.   We note that the latter result is applicable to magnetic fields that depend only on the scale.

\begin{lem}\label{lem:fieldgoestozerofast}
If $a$ is a real-valued $1$-form with flux depending only on the scale then $\sup_{|w|=m}4^m \bigl\| (a-a_m)\mathds{1}_{F_w(X)}\bigr\|_\Hil\to0$ as $m\to\infty$.
\end{lem}
\begin{proof}
From~\eqref{eqn:Hnormoffielddeponscale} it is apparent that $\bigl\| (a-a_m)\mathds{1}_{F_w(X)}\bigr\|_\Hil= \sum_m^\infty 4^n \beta_n^2$ for any $|w|=m$. Since there are $4^m$ such cells we see $4^m \bigl\| (a-a_m)\mathds{1}_{F_w(X)}\bigr\|_\Hil=\|a-a_m\|_\Hil\to0$.
\end{proof}
 
The following is then a direct consequence of Theorems~\ref{thm:cvgeofmagops} and~\ref{thm:findomMaga}.
\begin{cor}
For a real-valued $1$-form $a$ with flux depending only on the scale, $f\in\dom(\Mag^a)$ if and only if $\Mag^{a_m}_m f$ converges uniformly on $V_\ast$ to a continuous function, and in this case the continuous extension of this function to $X$ is $\Mag^a f$.
\end{cor}

We noted at the end of the previous section that if we construct a sequence of eigenfunctions of $\Mag^{a_m}_m$ on $V_m$ via spectral decimation then they converge on $V_\ast$.  Then $4^{m}\Mag^{a_m}_m f= 4^m z_m f$ converges on $V_\ast$  only if $4^m z_m$ converges.   This is not the case for most of the sequences of eigenvalues $\Mag^{a_m}_m$ we identified in Corollary~\ref{cor:Magamspectrum}, but it is true for sequences of a specfic type.

Let us write
\begin{equation*}
	S_m^\pm(w)
	= 1\pm \frac{1}{\sqrt{2}} \sqrt{1+(1-w)\cos2\beta_m}
	\end{equation*}
for the inverse branches of $R_m$. 

\begin{defn}
A sequence $z_m$ is admissible if it is of the form $z_m=S_m^{p_m}(z_{m-1})$ for $m>m_0$, where  $z_{m_0}$ is an eigenvalue of $\Mag^{a_{m_0}}_{m_0}$ and $p_m$ is a sequence with values in $\{-,+\}$, with the property that there is $n$ such that $p_m=-$ for all $m> n$.
\end{defn}

\begin{lem}\label{lem:admissseq}
If $z_m$ is an admissible sequence then $4^m z_m$ converges.
\end{lem}

\begin{proof}
First observe that $S_m^{\pm}$ preserves the interval $[0,2]$ and that the possible initial values are $\{0,z-,1,z+,2\}\subset[0,2]$.  It follows that all $z_m\in[0,2]$.  Moreover $S_m^{-}$ is contractive on $[0,2]$ and strictly contractive on $[0,2)$; in fact it is also strictly contractive on $[0,2]$ if $|\cos2\beta_m|<1$.   We consider only $m>n$, so need only look at iteration of $S_m^-$.  Using the fact that $\beta_m=o(4^{-m})$ we find that $|1-\cos 2\beta_m|=o(4^{-2m})$.  It is then easily checked that the contractive fixed point of $S_m^-$ is also $o(4^{-2m})$, and that the derivative there is $\frac14$ up to a factor in the interval $(1-4^{-m},1+4^{-m})$.  It follows that, for sufficiently large $m$, $z_m$ is close enough to $0$ that $4^{k}\bigl|S_{m+k}^-\circ\dotsm\circ S_{m+1}^-(z_m)\bigr|$ is within an interval of length a bounded multiple of $\Bigl( \Pi_m^{m+k} (1-4^{-j}),\Pi_m^{m+k}(1+4^{-j})\Bigr)$. Sending $k\to\infty$ we see that this interval may still be made arbitrarily small by taking $m$ sufficiently large, from which the result follows.
\end{proof}

\begin{rmk}
\end{rmk}In fact, the composition sequence $4^{m-n} S_m^-\circ\dotsm\circ S_{n+1}^-$ converges uniformly on a disc around $0\in\mathbb{C}$ to an analytic function $\Psi_n$ with $\Psi_n(0)=0$ and having derivative $\Psi_n'(0)=\prod_{n+1}^\infty \frac{\sqrt{2}\cos2\beta_m}{\sqrt{1+\cos2\beta_m}}$.  The $\beta_m$ are $o(4^{-m})$, so the product converges and  if there is no $m>n$ with $\cos 2\beta_m=0$ then $\Psi_n$ is invertible on a neighborhood of $0$; in particular the latter is true for all sufficiently large $n$.

\begin{lem}\label{lem:cvgeoff_m}
Suppose $z_m$ is an admissible sequence for $m\geq m_0$ and $f_{m_0}$ is an eigenfunction of $\Mag^{a_{m_0}}_{m_0}$ with eigenvalue $z_{m_0}$.  For $m> m_0$ define $f_m$ inductively by applying~\eqref{eqn:efnextension} to $f_{m-1}$, so $\Mag^{a_m}_m f_m=z_m f_m$. Then $f_m$ converges uniformly on $V_\ast$ to a continuous function $f$. 
\end{lem}

\begin{proof}
Using the explicit matrices given in Example~\ref{eg:simplefield} we find that if $x\in V_m\setminus V_{m-1}$ has neighbors $y_1,y_2\in V_{m-1}$ then $f_m(x) = \frac{1}{2(1-z_m)}(e^{i\beta_m}f(y_1)+e^{-i\beta_m}f(y_2))$. From Lemma~\ref{lem:admissseq} we know $z_m=O(4^{-m})$, and since $\beta_m=o(4^{-m})$ we conclude that the difference of $f_m$ across an edge of scale~$m$ is $O(2^{-m})$. The result follows.
\end{proof}

\begin{thm}\label{thm:spectofMaga}
If $z_{m}$, $m\geq m_0$ is an admissible sequence and $f_m$ is the corresponding sequence of eigenfunctions of $\Mag^{a_m}_m$ let $f$ be the continuous extension of $\lim_m f_m$ from $V_\ast$ to $X$.  Then $f$ is an eigenfunction of $\Mag^a$ with eigenvalue $z=\lim_m 4^m z_m$.  Conversely, if $f$ is an eigenfunction of $\Mag^a$ with eigenvalue $z$ then there is $m_0$ and a sequence $z_m$ such that $f$ and $z$ are obtained in this manner.
\end{thm}
\begin{proof}
Apply  Lemmas~\ref{lem:admissseq} and~\ref{lem:cvgeoff_m} to find that $4^m z_m$ converges and $f_m$ converges uniformly to $f$ on $V_\ast$.   Since $4^m\Mag^{a_m}_m f_m = 4^m z_m f_m$ one direction of the result follows by  Theorem~\ref{thm:cvgeofmagops}. The converse is a little more subtle; we proceed by proving that the eigenfunctions constructed as above are dense in $L^2(\mu)$.

Fix $0<\epsilon<1$, an eigenfunction $g$ of $\Mag^a$ with eigenvalue $\lambda$ and unit norm in $L^2(\mu)$, and a constant $Z>\lambda/\epsilon$.  Let $\Sigma$ denote the set of eigenvalues of $\Mag^a$ obtained by the spectral decimation procedure described above, let $\Sigma_Z=\Sigma\cap[-Z,0]$ and for $\sigma\in\Sigma$ let $\{\psi_{\sigma,j}\}_{j=1}^{J_{\sigma}}$ be an orthonormal basis for the corresponding eigenspace.  We show that 
\begin{equation}\label{eqn:spectofMaga0}
	\Bigl\| g - \sum_{\sigma\in\Sigma_Z} \sum_{j=1}^{J_\sigma}  \langle g,\psi_{\sigma,j} \rangle_{L^2(\mu)}\psi_{\sigma,j} \Bigr\|_{L^2(\mu)} < 3\epsilon
	\end{equation}
establishing that $\lambda\in\Sigma$.  The main difficulty in the proof is that the estimates take place in two spaces, neither of which is contained in the other.

It will be convenient for us to write $\tilde{\mu}_n=4^{-n}\mu_n$ to eliminate some factors of $4^n$.  In particular, if $\sigma\in\Sigma$ then it is a limit along an admissible sequnce, so there is $\sigma_n\to\sigma$ such that $\Mag^{a_n}_n \psi_{\sigma,j}=4^{-n}\sigma_{n} \psi_{\sigma,j}$ on the set $V_n$.  Moreover we may take $n$ so large that if $k>n$ then for all $\sigma\in\Sigma_Z$
\begin{gather}
	\bigl|\DF^{a_k}_k(\psi_{\sigma,j},g)-\DF^a(\psi_{\sigma,j},g)\bigr|<\epsilon, \label{eqn:spectoMaga1}\\
	\bigl|\DF^{a_k}_k(\psi_{\sigma,j},\psi_{\sigma,j'})-\DF^a(\psi_{\sigma,j},\psi_{\sigma,j'})\bigr|<\epsilon,  \label{eqn:spectoMaga2}\\
	\Bigl| \frac{\sigma_k}{\sigma} -1 \Bigr|<\epsilon,  \label{eqn:spectoMaga3}
	\end{gather}
where~\eqref{eqn:spectoMaga1} and~\eqref{eqn:spectoMaga2} are from Theorem~\ref{thm:cvgeofgraphapprox} and~\eqref{eqn:spectoMaga3} is from Lemma~\ref{lem:admissseq}.  From~\eqref{eqn:spectoMaga1} and~\eqref{eqn:spectoMaga3} we compute that as $k\to\infty$
\begin{align}
	\lefteqn{\bigl| \langle \psi_{\sigma,j},g \rangle_{l^2(\tilde{\mu}_k)} - \langle \psi_{\sigma,j},g \rangle_{L^2(\mu)} \bigr|} \quad& \notag\\
	&= \Bigl| \frac{1}{\sigma_k} \DF^{a_k}_k (\psi_{\sigma,j},g) - \frac1\sigma \DF^a ( \psi_{\sigma,j},g ) \Bigr|\notag \\
	&\leq \frac{1}{|\sigma|} \bigl| \DF^{a_k}_k(\psi_{\sigma,j},g) - \DF^a(\psi_{\sigma,j},g) \bigr| + \frac{1}{|\sigma|} \bigl|\DF^{a_k}_k(\psi_{\sigma,j},g) \bigr|\Bigl| \frac{\sigma}{\sigma_{k}} -1 \Bigr|\notag\\
	&\to  0, \label{eqn:spectoMaga4}
	\end{align}
and using~\eqref{eqn:spectoMaga2} and~\eqref{eqn:spectoMaga3} in the same manner shows that also
\begin{equation} \label{eqn:spectoMaga5}
	\bigl| \langle \psi_{\sigma,j},\psi_{\sigma',j'} \rangle_{l^2(\tilde{\mu}_k)} - \langle \psi_{\sigma,j},\psi_{\sigma',j'} \rangle_{L^2(\mu)} \bigr| \to 0
	\end{equation}
as $k\to\infty$.  Both limits are uniform for $\sigma,\sigma'\in\Sigma_Z$.  A similar argument shows $\|g\|_{l^2(\tilde\mu_k)}\to \|g\|_{L^2(\mu)}$.

The quantity~\eqref{eqn:spectofMaga0} may now be estimated using~\eqref{eqn:spectoMaga4} and~\eqref{eqn:spectoMaga5}.  Using orthonormality of the $\psi_{\sigma,j}$ in $L^2(\mu)$ and~\eqref{eqn:spectoMaga4} we may take $n$ so for $k>n$
\begin{align*}
	\Bigl\| g - \sum_{\sigma\in\Sigma_Z} \sum_{j=1}^{J_\sigma}  \langle g,\psi_{\sigma,j} \rangle_{L^2(\mu)}\psi_{\sigma,j} \Bigr\|_{L^2(\mu)} 
	&= \|g\|_{L^2(\mu)}^2 - \sum_{\sigma\in\Sigma_Z} \sum_{j=1}^{J_\sigma} \bigl| \langle g,\psi_{\sigma,j} \rangle_{L^2(\mu)}\bigr|^2 \\
	&\leq \|g\|_{l^2(\tilde{\mu}_k)} - \sum_{\sigma\in\Sigma_Z} \sum_{j=1}^{J_\sigma} \bigl| \langle g,\psi_{\sigma,j} \rangle_{l^2(\tilde{\mu}_k)}\bigr|^2 + \epsilon
	\end{align*}
Moreover if $Q_k$ is the matrix with entries $\langle \psi_{\sigma,j},\psi_{\sigma',j'} \rangle_{l^2(\tilde{\mu}_k)}$ (for $\sigma,\sigma'\in\Sigma_Z$ and all relevant $j,j'$ for each $\sigma,\sigma'$) then by~\eqref{eqn:spectoMaga5} $Q_k$ converges to the identity.  Since $\tilde\psi_{\sigma,j}=\sum_{\sigma',j'} Q_k \psi_{\sigma',j'}$ is an orthonormal set in $l^2(\tilde\mu_k)$ we use both facts to conclude that for large enough $k$
\begin{align}
	\Bigl\| g - \sum_{\sigma\in\Sigma_Z} \sum_{j=1}^{J_\sigma}  \langle g,\psi_{\sigma,j} \rangle_{L^2(\mu)}\psi_{\sigma,j} \Bigr\|_{L^2(\mu)} 
	&\leq \|g\|_{l^2(\tilde{\mu}_k)} - \sum_{\sigma\in\Sigma_Z} \sum_{j=1}^{J_\sigma} \bigl| \langle g,\tilde\psi_{\sigma,j} \rangle_{l^2(\tilde{\mu}_k)}\bigr|^2 + 2\epsilon \notag\\
	&= \Bigl\| g - \sum_{\sigma\in\Sigma_Z} \sum_{j=1}^{J_\sigma} \langle g,\tilde\psi_{\sigma,j} \rangle_{l^2(\tilde{\mu}_k)}\tilde\psi_{\sigma,j} \Bigr\|_{l^2(\tilde\mu_k)}^2  +2\epsilon \label{eqn:spectoMaga6}
	\end{align}

At this juncture we recall that $\tilde\psi_{\sigma,j}$ is an eigenfunction of $\Mag^{a_k}_k$ with eigenvalue $\sigma_k$ when treated an element of $l^2(\tilde\mu_k)$. Since  the eigenfunctions of $\Mag^{a_k}_k$ are complete in $l^2(\tilde\mu_k)$ the expression~\eqref{eqn:spectoMaga6} is the projection onto those eigenfunctions of $\Mag^{a_k}_k$ for which the corresponding eigenvalue gives rise to elements of $\Sigma$ of size larger than $Z$.  However~\eqref{eqn:spectoMaga3} shows that any such eigenvalue must be larger than $Z/(1-\epsilon)$.  Using this and the observation
\begin{equation*}
	\langle g,\tilde\psi_{\sigma,j} \rangle_{l^2(\tilde{\mu}_k)}
	= \frac{1}{|\sigma_k|} \langle g, \Mag^{a_k}_k \tilde\psi_{\sigma,j} \rangle_{l^2(\tilde{\mu}_k)} 
	= \frac{1}{|\sigma_k|} \langle  \Mag^{a_k}_k  g,\tilde\psi_{\sigma,j} \rangle_{l^2(\tilde{\mu}_k)}
	\end{equation*}
we have at last
\begin{align*}
	\Bigl\| g - \sum_{\sigma\in\Sigma_Z} \sum_{j=1}^{J_\sigma}  \langle g,\psi_{\sigma,j} \rangle_{L^2(\mu)}\psi_{\sigma,j} \Bigr\|_{L^2(\mu)} 
	&\leq 2\epsilon+ \sum_{|\sigma_k|>Z/(1-\epsilon)} \bigl| \langle g,\tilde\psi_{\sigma,j} \rangle_{l^2(\tilde{\mu}_k)} \bigr|^2 \\
	&= 2\epsilon+ \sum_{|\sigma_k|>Z/(1-\epsilon)} \frac{1}{|\sigma_k|^2} \bigl| \langle \Mag^{a_k}_k g,\tilde\psi_{\sigma,j} \rangle_{l^2(\tilde{\mu}_k)} \bigr|^2 \\
	&\leq 2\epsilon+ \frac{(1-\epsilon)^2}{Z^2} \bigl\| \Mag^{a_k}_k g \bigr\|_{l^2(\tilde\mu_k)}^2
	\end{align*}
however $\bigl\| \Mag^{a_k}_k g \bigr\|_{l^2(\tilde\mu_k)}^2=\bigl\| 4^k\Mag^{a_k}_k g \bigr\|_{l^2(\mu_k)}^2$, and our field satisfies Lemma~\ref{lem:fieldgoestozerofast}, so in light of Theorem~\ref{thm:findomMaga} we have $\bigl\| \Mag^{a_k}_k g \bigr\|_{l^2(\tilde\mu)}^2\to \| \Mag^a g\|_{L^2(\mu)}=\lambda^2\|g\|_{L^2(\mu)}^2=\lambda^2$.  By assumption $\lambda<\epsilon Z$, so~\eqref{eqn:spectofMaga0} holds and the proof is complete.
\end{proof}

\section{Numerical results for a uniform field}\label{sec:numerics}

It is physically natural to consider the case when the magnetic field through the fractal is uniform, and therefore the flux through each cell is proportional to the area of the cell.  Of course, when $X$ is thought of as an abstract self-similar set there is no notion of the area of a cell, so we make the assumption that the area of a cell of scale~$m$ is $Cr^m$, for some constants $C$ and $0<r<1/4$, where the  latter restriction is based on the idea that there are four cells of scale $m+1$ in each cell of scale~$m$.   Note that for a smaller range of $r$ we presented an embedding of $X$  into $\mathbb{R}^2$ at the beginning of Section~\ref{sec:analysis} in which the area of each scale~$m$ cell is $2s^{m-1}$, where $0<s\leq 1/8$ is a fixed factor.

Our first task is to determine the values in the sequence $\beta_n$ used in Definition~\ref{def:fluxindepofscale} that correspond to a uniform field of the above type.  Observe that for a given~$m$ the flux through cells of scale~$m$ depends only on $\beta_n$ for $n\geq m$ because the contributions from the $n<m$ are gauge fields for cells of scale~$m$.  Since there are $4^{n-m}$ cells of scale $n$ in a cell of scale~$m$ and each contributes flux $4\beta_n$ the total flux through such a cell, assuming $\beta_n$ as in the statement of the lemma, is  $4\sum_{n=m}^\infty 4^{n-m}\beta_n =4\sum_0^\infty 4^n\beta_{m+n}$.  Evidently if $\beta_m=\beta r^m$ for $r<1/4$ then the flux through an $m$~cell is $Cr^m$ for $C=4\beta\sum_0^\infty (4r)^n$.  This is a special case of a field that depends only on the scale, so from Theorem~\ref{thm:Specsimonelevel}, equation~\eqref{eqn:RmapforDLF} and Corollary~\ref{cor:Magamspectrum} we should set
\begin{gather*}
	R_m= \frac{-2z^2+4z-1+\cos(2\beta r^m)}{\cos(2\beta r^m)}\\
	\mathcal{R}_{k,m} = R_k\circ \dotsm \circ R_m
	\end{gather*}
at which point the spectrum of $\Mag^{a_m}_m$ is $\{1\}\cup\Bigl(\cup_{k=2}^m\mathcal{R}_{k,m}^{-1}\{1\}\Bigr)$ with multiplicity $\frac43(4^k-1)$ for points in $\mathcal{R}_{k,m}^{-1}$ and $\mult_{m}(1)=\frac13(4^m+2)$.   Figure~\ref{fig:graphLapspect24}  shows the dependence of spectra of this type on the magnetic field strength $\beta$ when $r=.24$ is close to the limiting value of $1/4$, while Figure~\ref{fig:graphLapspect1} shows the same dependence when $r=0.1$.  Two levels of approximation ($m=5$ and $m=7$) are shown to emphasize the manner in which the graph spectral values accumulate on an attractor.

\begin{figure}[t!]
\begin{centering}
\includegraphics[viewport= 55 200 490 580,clip, width=6cm]{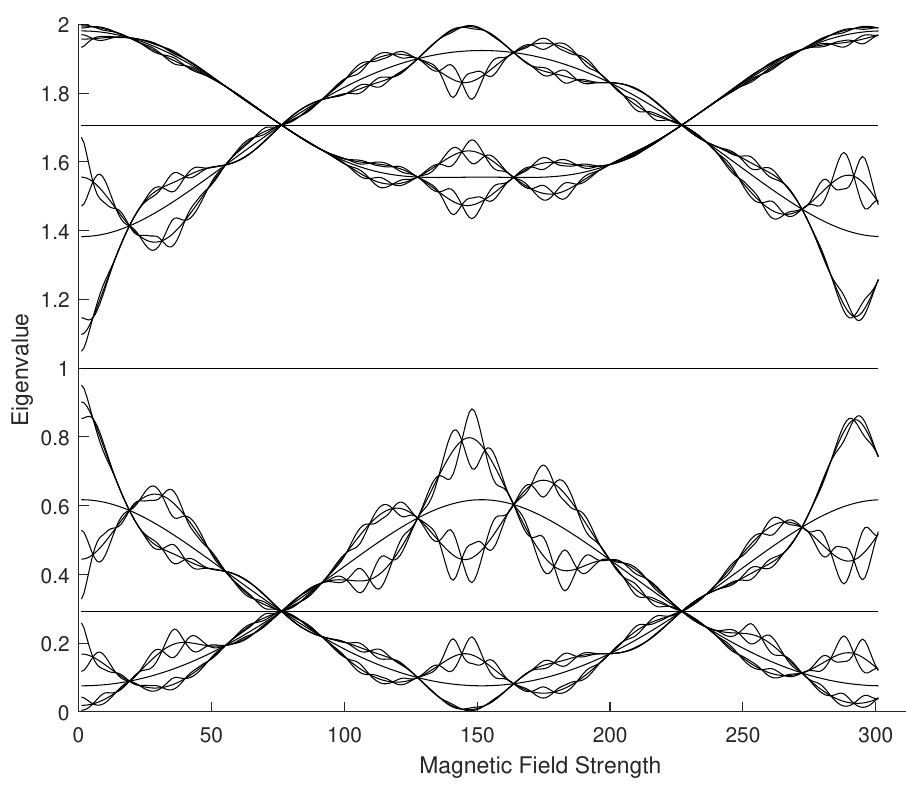}\hfill
\includegraphics[viewport= 55 200 490 580, clip, width=6cm]{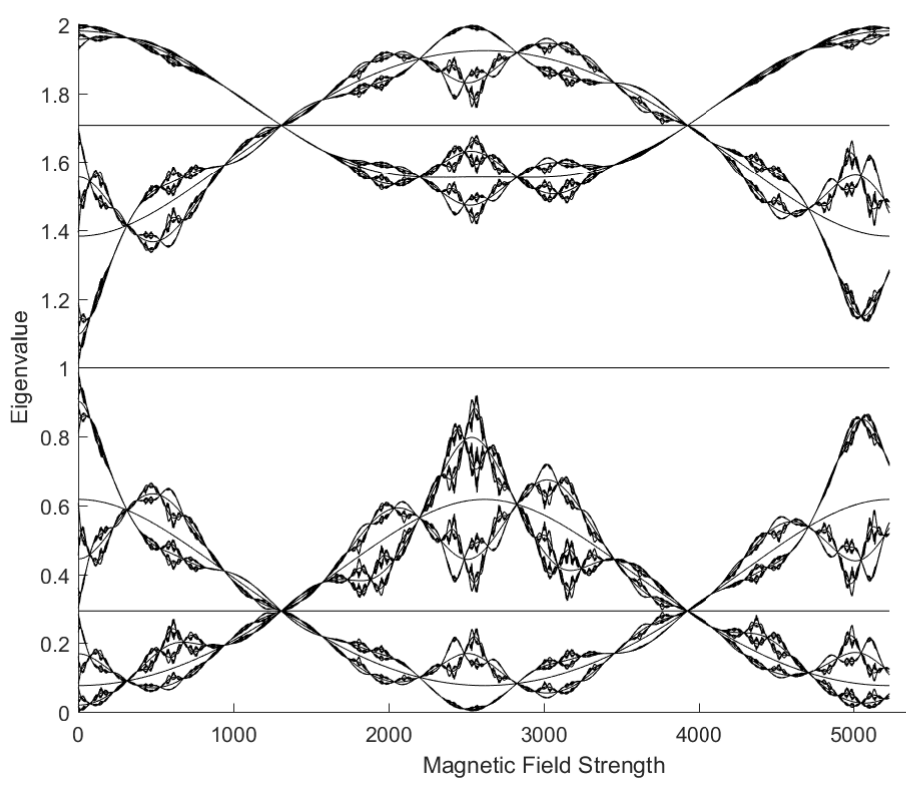}
\end{centering}
\caption{Spectral values versus magnetic field strength for $\Mag_m^{a_m}$ of level $m=5$ (left) and $m=7$ (right) when $r=0.24$.}\label{fig:graphLapspect24}
\end{figure}
\begin{figure}[t!]
\begin{centering}
\includegraphics[viewport= 55 180 560 600,clip, width=6cm]{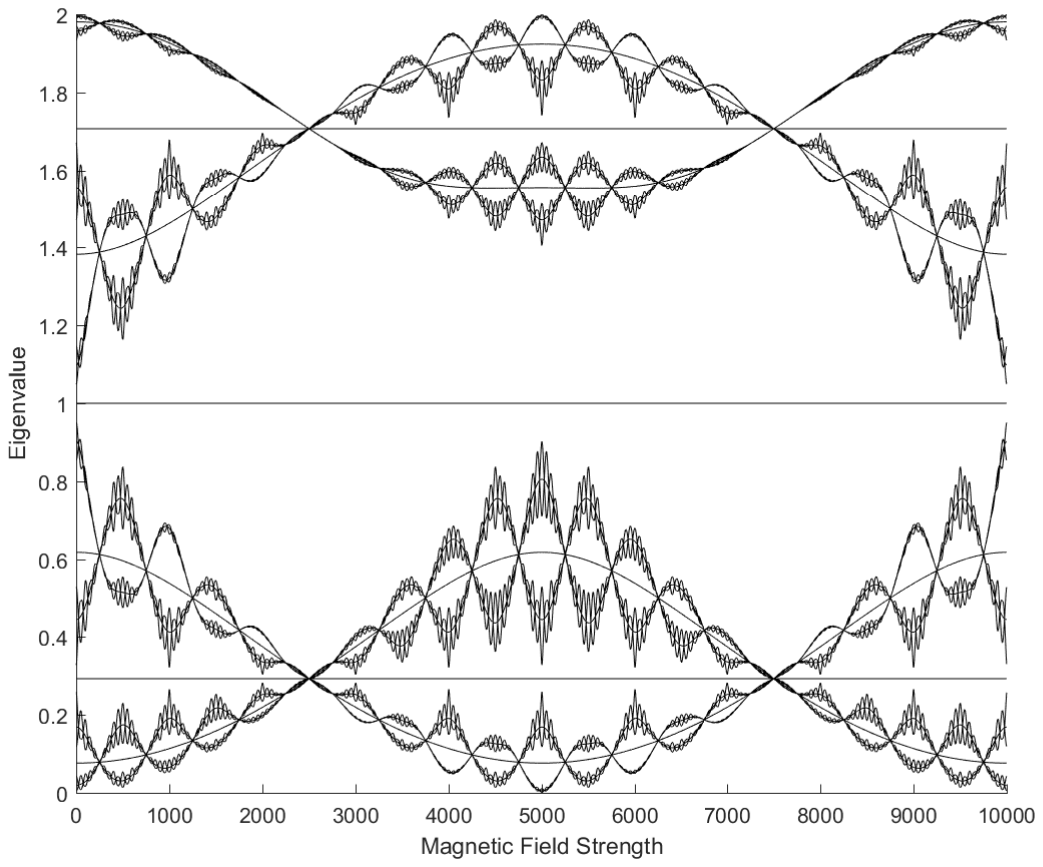} \hfill
\includegraphics[viewport=55 180 560 600,clip, width=6cm]{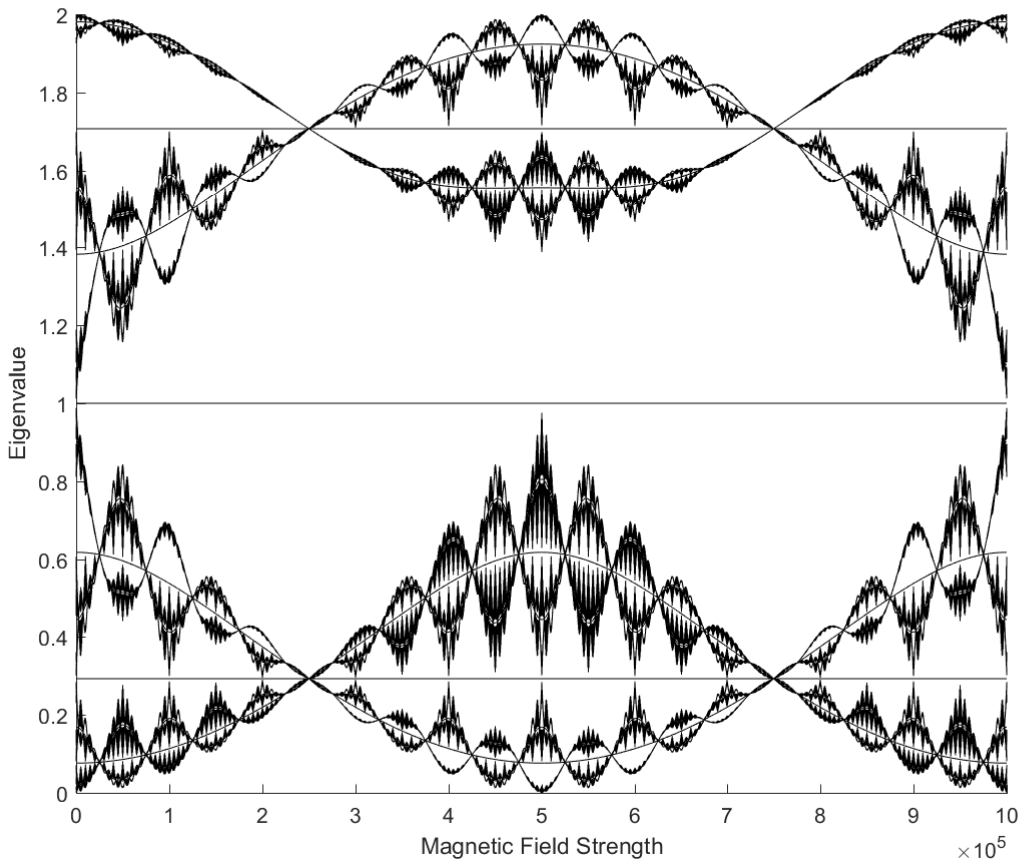}
\end{centering}
\caption{Spectral values versus magnetic field strength for $\Mag_m^{a_m}$ of level $5$ (left) and $7$ (right) when $r=0.1$.}\label{fig:graphLapspect1}
\end{figure}

According to Theorem~\ref{thm:spectofMaga} the spectrum of the corresponding magnetic operator $\Mag^a$ may be obtained from the spectra of $\Mag^{a_m}_m$ by taking a renormalized limit.  Numerical results show the first few eigenvalues in the spectrum of $\Mag^a$ are well approximated by taking quite small values of $m$. Figure~\ref{fig:fractalspectrum24} shows the dependence of the first $32$ eigenvalues of the fractal magnetic operator as a function of $\beta$ when computed using $m=11$; in this example $r=0.24$.  The same graph for $r=0.1$ is in Figure~\ref{fig:fractalspectrum1}.
\begin{figure}[!h]
\begin{centering}
\includegraphics[viewport= 50 200 555 555, clip,width=12cm]{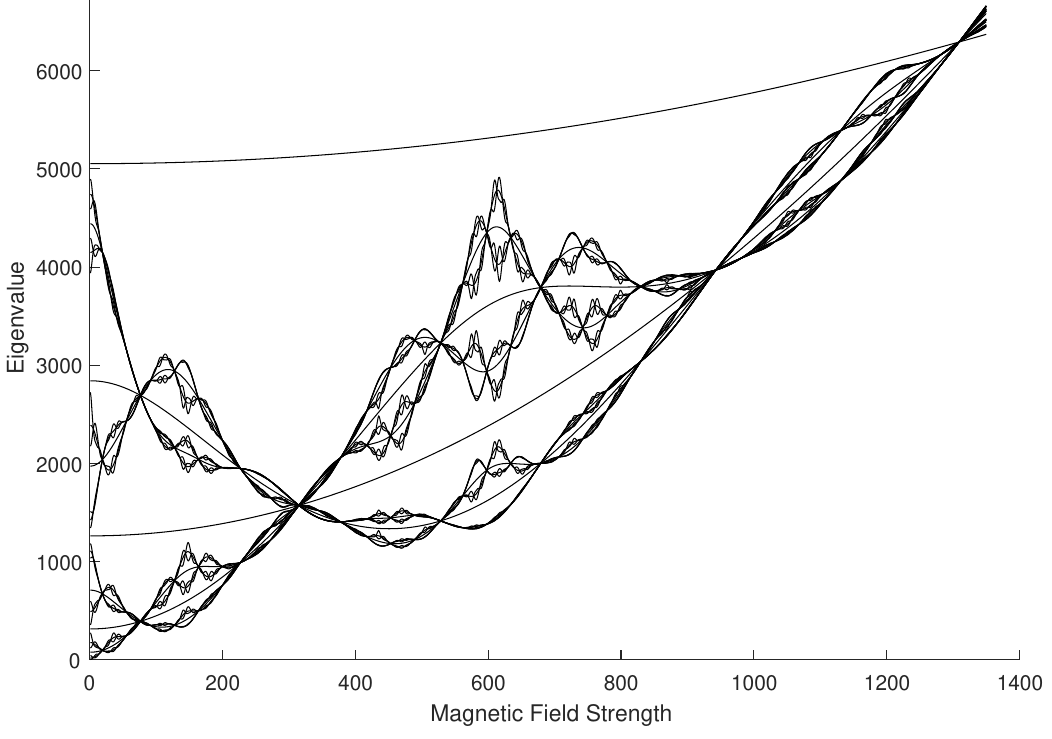}
\end{centering}
\caption{Spectral values vs. magnetic field strength for $\Mag^{a}$ with $r=0.24$.}\label{fig:fractalspectrum24}
\end{figure}
\begin{figure}[H]
\begin{centering}
\includegraphics[viewport= 50 200 555 545,clip, width=12cm]{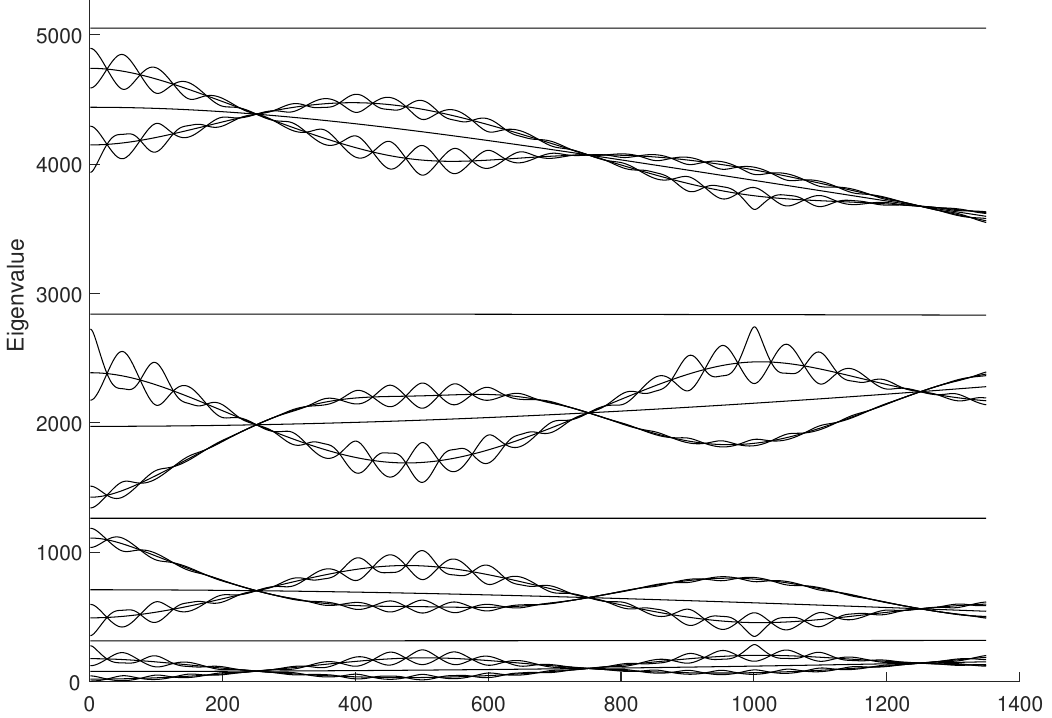}
\end{centering}
\caption{Spectral values vs. magnetic field strength for $\Mag^{a}$ with $r=0.1$.}\label{fig:fractalspectrum1}
\end{figure}

\newpage
\bibliography{DiamondMagSpecDec}
\bibliographystyle{amsplain}

\end{document}